\documentclass{article}
\usepackage{graphicx}
\usepackage{circuitikz}
\usepackage{amssymb}
\usepackage{answers}
\usepackage{setspace}
\usepackage{graphicx}
\usepackage{multicol}
\usepackage{mathrsfs}
\usepackage{hyperref}
\usepackage{quoting}
\usepackage{xcolor}
\usepackage{amsmath,amsthm,amssymb}
\usepackage{tikz}
\usepackage{quiver}
\usepackage{natbib}
\usepackage{array, booktabs}
\usepackage{ tipa }

\usetikzlibrary {arrows.meta,bending,positioning, fit}

 \newtheorem{question}{Question}
  \newtheorem{conjecture}{Conjecture}
\newtheorem{observation}{Observation}
 \newtheorem{fact}{Fact}
 \newtheorem{theorem}{Theorem}
\newtheorem{lemma}{Lemma}

\newtheorem{proposition}{Proposition}

\theoremstyle{definition}
\newtheorem{definition}{Definition}

\theoremstyle{definition}

\theoremstyle{definition}

\def\Ind#1#2{#1\setbox0=\hbox{$#1x$}\kern\wd0\hbox to 0pt{\hss$#1\mid$\hss}
\lower.9\ht0\hbox to 0pt{\hss$#1\smile$\hss}\kern\wd0}

\def\ind{\mathop{\mathpalette\Ind{}}}

\def\notind#1#2{#1\setbox0=\hbox{$#1x$}\kern\wd0
\hbox to 0pt{\mathchardef\nn=12854\hss$#1\nn$\kern1.4\wd0\hss}
\hbox to 0pt{\hss$#1\mid$\hss}\lower.9\ht0 \hbox to 0pt{\hss$#1\smile$\hss}\kern\wd0}

\def\nind{\mathop{\mathpalette\notind{}}}

\title{Reducing stable forking dependence to finitely many pregeometries}
\author{Scott Mutchnik\footnote{This project was supported by the NSF under Grant No. DMS-2303034.}}
\date{July 2026}

\begin{document}

\maketitle

\begin{abstract}

We show that one of the main cases of the stable forking conjecture, stability of the forking relation over a base in a finite-rank supersimple theory, is determined by finitely many pregeometries in each rank. This case of the stable forking conjecture has long had an implicitly well-known pregeometric interpretation: there is a set of matroids $\mathcal{G}_{n}$ such that the forking instability in rank $n$ is equivalent to the pregeometry on some rank-one partial type (over a finite set) embedding a matroid in $\mathcal{G}_{n}$. Our contribution is to show that this set of matroids $\mathcal{G}_{n}$, determining based forking stability in rank $n$, can be chosen to be finite.

The main part of our proof was already accomplished in rank $3$ by Peretz, but does not extend as stated to higher ranks (and may or may not directly extend in a weaker sense to higher ranks, by shrinking terms). However, we obtain a sufficient substitute for Peretz's work in ranks $n > 3$: we turn Peretz's original universal result into an existence theorem. The rest of our proof refines an argument from multi-experiment parameter identifiability, originating from work in applied model theory by Li, Meshkat, Ovchinnikov, Pillay, Pogudin and Scanlon.

\end{abstract}

\section{Introduction}

One of the most important problems in model theory is the \textit{stable forking conjecture}, which hypothesizes an especially strong connection between simplicity and the even more powerful phenomenon of stability. Stability and simplicity are already deeply connected even in the absence of the stable forking conjecture: they are tied together by independence phenomena within the classification theory of first-order logical theories. And while the origins of classification theory lie squarely within the context of stable theories, the rich web of interactions between stability theory and classes of \textit{unstable} theories, such as simple theories, has become one of the field's fundamental pillars.

The original problem prompting the invention of classification theory mainly had to do with stable theories: given a first-order theory $T$, and an infinite cardinal $\kappa$, how many models of size $\kappa$ does $T$ have (up to isomorphism)? The initial breakthrough on this question was due to Morley (\cite{M65}), who showed that a theory which is $\kappa$-categorical for \textit{some} uncountable cardinal $\kappa$, meaning that it has exactly one model of size $\kappa$, must be $\kappa'$-categorical for \textit{every} uncountable cardinal $\kappa'$. Yet it was Shelah (\cite{Sh90}) who established classification theory's central role within model theory by attacking this model-counting question more comprehensively. His central insight is to isolate, from among the theories whose model-counting status is nontrivial, the class of \textit{stable theories}: theories for which no formula has the \textit{order property}, meaning that there is no formula $\varphi(x, y)$ with a sequence $\{a_{i}, b_{i}\}_{i < \omega}$ such that $\models \varphi(a_{i}, b_{j})$ exactly when $i > j$. When a first-order theory $T$ is unstable, the problem of counting the number of models $I(T, \kappa)$ of $T$ of cardinality $\kappa$ becomes trivial, for uncountable $\kappa$: $T$ will just have the maximum number of models $I(T, \kappa) = 2^{\kappa}$.\footnote{But the problem of counting \textit{countable} models of a first-order theory, even for unstable theories, remains an active and exciting area of research, from Vaught's conjecture to the lowness problem posed in \cite{palacin2012omega} to the Koponen conjecture on countably categorical $n$-ary theories, the last of which was resolved by Baldwin, Freitag and the author in \cite{LongKoponenConjecturePreprint}.} By contrast, the model-counting problem in stable theories is extremely complex, but nonetheless tractable due to the structure theory that Shelah develops for stable theories.

Just as the rank of a definable set, introduced by Morley for uncountably categorical theories, generalizes the dimension of an algebraic variety, Shelah's structure theory defines an independence relation in stable theories that also has its counterpart in algebraic geometry. Specifically, Shelah's \textit{forking independence} $A \ind_{C} B$, between sets within a model of a theory, generalizes the algebraic freeness of fields $A$ and $B$ over a common subfield $C$. Shelah shows that forking independence is symmetric in any stable theory, meaning that $A \ind_{C} B$ implies $B \ind_{C} A$. Moreover, Shelah shows that stability is characterized by \textit{stationarity} of forking independence, or uniqueness of nonforking extensions over models. Shelah's structure theory for stable theories serves as the main ingredient for his proof of the main gap theorem, which gives a sweeping classification of first-order theories in terms of their numbers of size-$\kappa$ models $I(T, \kappa)$, a classification that is ultimately completed for uncountable $\kappa$ in \cite{HHL00}. The modern incarnation of Shelah's characterization of stability in terms of forking independence $A \ind_{C} B$ is established by Harnik and Harrington in \cite{HarnikHarrington1984}.

Because Shelah's model-counting program breaks down outside the class of stable theories, the main informal research program of interest today arising out of Shelah's work is to classify the more complex, unstable theories. Shelah initiates this program in \cite{Sh90} by introducing part of what is now considered to be the classical hierarchy of unstable first-order theories. This classical hierarchy was expanded in \cite{She95}, completed by Dzamonja and Shelah in \cite{DS04} (as referenced earlier in \cite{Sh99}, and further investigated in \cite{SU08}), and given its standard visual representation by Conant in \cite{FD}. At the historic core of this hierarchy of unstable theories is the class of \textit{simple theories}: theories where no formula has the \textit{tree property}, analogous to no formula in a stable theory having the order property. While every stable theory is simple, not every simple theory is stable. Nonetheless, it is within the class of simple theories that one of the most powerful approaches to classifying unstable theories originates: describing the structure of unstable theories using phenomena from stability theory itself, including the theory of independence.

This approach begins in celebrated work of Kim (\cite{KP98}), who shows that forking independence is not just symmetric in every stable theory: it is even symmetric in every simple theory. In fact, just like stationary of forking independence gives a complete characterization of stability, \textit{symmetry} of forking independence gives a complete characterization of simplicity (\cite{KP99}). Though in general, forking independence is not stationary in simple theories, Kim and Pillay show (\cite{KP99}) that forking independence in simple theories even satisfies its own analogue of stationarity: $3$-amalgamation or the independence theorem, whereby if $a \ind_{M} b_{1}$, $a \ind_{M} b_{2}$, $a \equiv_{M} a'$, and $b_{2} \ind_{M} b_{1}$, there is $a'' \ind_{M} b_{1}b_{2}$ with $a''\equiv_{Mb_{1}}a$, $a''\equiv_{Mb_{2}} a'$.

From these main results of simplicity theory, the conclusion is unavoidable that there is a deep structural connection between stability and simplicity. Yet this conclusion leaves open the possibility that the connection between stability and simplicity is more than just structural. In this scenario, there is a direct connection between, on one hand, the forking independence relation in simple theories, and, on the other, stable formulas themselves. This potential direct connection is made precise by the \textit{stable forking conjecture}, originating from a personal communication of Hart, Kim and Pillay cited in \cite{Kim01}:

\begin{conjecture}\label{stable forking conjecture}

\emph{(Stable forking conjecture)}

In a simple $\mathcal{L}$-theory $T$, let $a \nind_{C} b$. Then there is a formula $\varphi(x, \overline{b}) \in \mathrm{tp}(a / Cb)$ such that $\varphi(x, \bar{b})$ forks over $C$, and such that the $\mathcal{L}$-formula $\varphi(x, y)$ is stable (i.e. does not have the order property).

\end{conjecture}

Some further motivation for the stable forking conjecture comes from the consequences of its versions in larger classes of theories, including in $\mathrm{NSOP}_{1}$ theories and $\mathrm{NTP}_{2}$ theories. In $\mathrm{NSOP}_{1}$ theories, Gomez shows (\cite{GomezRosyTheories}) that the \textit{stable Kim-forking conjecture} implies that there is no strictly $\mathrm{NSOP}_{1}$ rosy theory, which resolves a problem of \cite{D19}, \cite{K21}. In $\mathrm{NTP}_{2}$ theories, proving the \textit{dependent forking conjecture} (posed as the ``dependent dividing" conjecture in \cite{Che14}) would have consequences for other open problems.  It is observed by Chernikov in a personal communication, as well as independently by the author, that the dependent forking conjecture over \textit{sets} implies that $\mathrm{NTP}_{2}$ is equivalent to resilience, which would answer a question posed in \cite{BYC14}.\footnote{It is shown in \cite{BYC14} that all dependent formulas are resilient, but resilient dividing over sets can be shown to imply resilience. Chernikov pointed out in the same personal communication that the latter fact relies on the resilient forking conjecture over sets, and is not clear from just the resilient forking conjecture over models.} Additionally, the dependent forking conjecture over models implies that $\mathrm{NSOP} \cap \mathrm{NTP}_{2} = \mathrm{simple}$\footnote{Suppose that $T$ is $\mathrm{NSOP}$ and $\mathrm{NTP}_{2}$, and the dependent dividing conjecture holds over models. Then every instance of dependence over a model is exhibited by a dependent formula. But it follows from Shelah's proof in \cite{Sh90} of the identity $\mathrm{NSOP}\cap \mathrm{NIP} = \mathrm{stable}$ that every dependent formula in an $\mathrm{NSOP}$ theory is stable. So every instance of dependence over a model is exhibited by a stable formula (the conclusion of the stable forking conjecture). But this is well-known to imply simplicity. (For example, the proof of symmetry of forking from \cite{KP98} works in this case. And by \cite{KP99}, symmetry of forking implies simplicity.)}, which would resolve a central open problem of classification theory posed in \cite{Che14}, with roots even earlier in a purported counterexample in \cite{Sh90}.

It is worth noting some additional consequences from the literature of special assumptions on theories strengthening the stable and dependent forking conjectures. Onshuus (\cite{Ons06}) shows that, in theories satisfying the conclusion of the \textit{strong stable forking conjecture}\footnote{As shown by Kim and Pillay (\cite{kim2001around}), the strong stable forking conjecture is known to be false in general.}, forking must be equal to thorn-forking. (This is connected to work of Ealy and Onshuus (\cite{ealy2016thorn}), where they prove an equivalence between thorn-forking and a parametrized version of forking via a stable formula.) Palacín also shows that every countably categorical simple theory satisfying the conclusion of the strong stable forking conjecture is low. Moreover, Takahashi (\cite{TakahashiConsequences}) applies an assumption on theories based on the conclusion of the dependent forking conjecture to subadditivity of burden in $\mathrm{NTP}_{2}$ theories.

As Casanovas and Potier observe, ``there is not much progress on the stable forking conjecture" (\cite{casanovaspotier2018}). The first result is due to Kim (\cite{Kim01}), who proves the conclusion of the stable forking conjecture in one-based theories with elimination of hyperimaginaries.\footnote{Brower and Hill (\cite{BrowerHillHyperimaginaries}) give a shorter proof that finite-rank supersimple theories with \textit{weak} elimination of hyperimaginaries satisfy the conclusion of the stable forking conjecture, without relying on the elimination of hyperimaginaries in supersimple theories proven in \cite{BuechlerPillayWagner2001}.} Peretz (\cite{Per06}) proves the conclusion of the stable forking conjecture relative to a base between types of $\mathrm{SU}$-rank $2$ in a countably categorical theory. He also investigates the $\mathrm{SU}$-rank-$3$ case of the stable forking conjecture; we will discuss his work in rank $3$ later in this paper. Brower (\cite{Brower2012}) improves on Peretz's results in rank $2$, eliminating the assumption of countable categoricity, and proving the conclusion of the stable forking conjecture over a base between a type of rank $2$ and a type of any finite rank. Palacín and Wagner (\cite{PW13}) prove the conclusion of the stable forking conjecture in countably categorical, supersimple $\mathrm{CM}$-trivial theories, and Casanovas and Potier (\cite{casanovaspotier2018}) show that the conclusion of the stable forking conjecture is preserved under constructing $T^{\mathrm{eq}}$. Baldwin and Freitag, in joint work with the author of this paper (\cite{LongKoponenConjecturePreprint}), also make progress on the \textit{simple Kim-forking conjecture} in $\mathrm{NSOP}_{1}$ theories. It is shown there that an infinite-variable version of simplicity of the Kim-forking relation holds in general, that a finite-variable version holds assuming finiteness of $F_{Mb}$, and that the simple Kim-forking conjecture itself is true (with nontrivial instances) under the assumption of definable Morley sequences.

Yet outside of the relatively narrow cases of $\mathrm{SU}$-rank-$2$ types and $\mathrm{CM}$-trivial theories, past progress on the stable forking conjecture remains limited, and we will not lose much difficulty concentrating on some central special cases. To start, like Peretz (\cite{Per06}) and Brower (\cite{Brower2012}), we can concentrate on the stable forking conjecture relative to a base, or equivalently, over the empty set. And instead of the full stable forking conjecture over a base, we can instead consider stability of the forking relation itself over a base, as in Palacín and Wagner (\cite{PW13}) or as with simplicity of the Kim-forking relation in \cite{LongKoponenConjecturePreprint}. Even under especially strong assumptions, such as between types of $\mathrm{SU}$-rank $3$ in countably categorical theories, stability of the forking relation over a base remains very much open. In this paper, we will be interested in stability of the forking relation over a base in theories of finite $\mathrm{SU}$-rank.

The goal of this paper will be to connect this core case of the stable forking conjecture to another research program inspired by Shelah's original development of stability theory: \textit{geometric stability theory}. The paradigmatic object in geometric stability theory is a \textit{strongly minimal set}: a model of a theory where every set definable in one variable, or the complement of that set, is finite. As shown by Baldwin and Lachlan in \cite{baldwin1971strongly}, strongly minimal sets come equipped with a pregeometry, or the matroid given by the algebraic closure operation on that set: for example, vector spaces with the linear span operation, or algebraically closed fields with the field-theoretic algebraic closure. And according to the motivating conjecture of geometric stability theory, Zilber's proposed trichotomy for pregeometries (\cite{zilber1986structural}), these two examples are essentially the only nontrivial instances of the pregeometry arising from a strongly minimal set. Specifically, Zilber posits that the pregeometry on a strongly minimal set is always (1) ``set-like" or trivial, (2) ``vector space-like" or locally modular, or (3) ``field-like," meaning that the theory of the strongly minimal set interprets an algebraically closed field. The first main breakthrough on this conjecture is due to Hrushovski (\cite{Hrustrongmin}), who shows that the Zilber trichotomy is false. Yet more recent developments in geometric stability theory have shown that the Zilber trichotomy is true for some very rich classes of pregeometries. Resolving a classical open problem, Castle (\cite{castle2024restricted}) shows that the Zilber trichotomy is true for pregeometries arising from reducts of a characteristic-zero algebraically closed field, and Castle, Hasson and Ye (\cite{CastleTrichotomy}) extend Castle's results to the case of positive characteristic.

It is implicitly well-known, as a standard application of equidominance tools from geometric stability theory (see, e.g., \cite{P96}), that stability of the forking relation over a base can be interpreted in terms of pregeometries. Here, we consider embeddings of abstract pregeometries (i.e., matroids) into the pregeometry given by the algebraic closure on an $\mathrm{SU}$-rank-$1$ partial type.

\begin{fact}\label{standard pregeometric interpretation}

Let $n < \omega$. Then there is a set $\mathcal{G}_{n}$ of (infinite) pregeometries, depending only on $n$, such that the following are equivalent for any finite-rank supersimple theory $T$:

    \begin{itemize}
        \item In $T$, the forking relation is stable over a base between types of rank $n$.
        \item No pregeometry in $\mathcal{G}_{n}$ embeds into the pregeometry on an $\mathrm{SU}$-rank-$1$ partial type over a finite set in $T$.
    \end{itemize}
    
\end{fact}

However, this interpretation alone is not very deep: $\mathcal{G}_{n}$ is just some infinite set. So for the correspondence with pregeometries to give us actual nontrivial information, we will need to get additional control over this set of pregeometries $\mathcal{G}_{n}$ that describes all of the ways in which stability of the forking relation can fail.

In this paper, we will show that the pregeometric interpretation of forking stability has real descriptive power. Specifically, we will show that we can reduce stability of the forking relation over a base, between types of any fixed rank, to \textit{finitely many}  pregeometries. So an important and relatively general family of cases of the stable forking conjecture becomes more tractable, reducing to checking finitely many pregeometries in each rank. Our main theorem of this paper will be to show:

\begin{theorem}\label{main theorem}

Let $n < \omega$. Then there is a \emph{finite} set $\mathcal{G}_{n}$ of (infinite) pregeometries, depending only on $n$, such that the following are equivalent for any finite-rank supersimple theory $T$:

    \begin{itemize}
        \item In $T$, the forking relation is stable over a base between types of rank $n$.
        \item No pregeometry in $\mathcal{G}_{n}$ embeds into the pregeometry on an $\mathrm{SU}$-rank-$1$ partial type over a finite set in $T$.
    \end{itemize}

\end{theorem}

When $n = 3$, which is the lowest-rank case where stability of the forking relation over a base remains open, the technical core of the proof of this theorem was already proven by Peretz in \cite{Per06}. Peretz shows that in \textit{every} case where $\{a_{i}, b_{i}\}$ is a $C$-indiscernible sequence exhibiting instability of the forking relation $\nind_{C}$ over a base $C$ between types of rank $3$, meaning that $a_{i} \nind_{C} b_{j}$ exactly when $i > j$, the sequence $\{b_{i}\}_{i < \omega}$ is an independent or Morley sequence over a set of small $\mathrm{SU}$-rank. Specifically, in this case, $\{b_{i}\}_{i \geq 2}$ is a Morley sequence over $b_{0}b_{1}$.

Peretz hints in this same paper at generalizations of his results to higher ranks, potentially involving induction:

\begin{quoting}
    Simplifications of the proof can be achieved (for instance, if one is only interested in stable forking in U-rank 2, extensive simplifications are possible), but the following proof provides a better background for generalization, where the same basic structure could be repeated for higher SU-ranks (though the situation becomes more complicated).
\end{quoting}

\begin{quoting}

 The fact that we got a precise formula (moreover that it is forking) is very useful when one comes to extend this result to higher SU-ranks (e.g., for a proof by induction), and allows the result itself to be used and not just the methodology of the proof.

 Remark. (1) I will mention that we could prove Claim 2.3 using the independence theorem over Lascar strong types. Such a proof, though longer in this case, allows for certain generalizations to higher SU-ranks by changing our elements to tuples.
\end{quoting}

However, according to a personal communication with Tom Scanlon, it is understood that Peretz's generalizations did not work as expected.\footnote{Moreover, the author, on reading Peretz's proof in rank $3$, was confused about one part of the argument: the justification of why $\mathrm{SU}(a/xb) =2$ in the third-to-last paragraph. Further conversations with Tom Scanlon about Diego Bejarano's investigations into Peretz's proof corroborate this confusion. Since we do not need Peretz's result in rank $3$ for the proof of any of our own results, and any errors in Peretz's arguments would not negatively affect the originality of the results of this paper, we will address our paper to Peretz's results as stated and will make no claim that any such errors are unfixable. If we cannot assume Peretz's theorem, then our work in Section \ref{an existential Peretz theorem} would have a different relationship to Peretz's statement. Under the assumption that Peretz's results are correct, we discuss how the original universal version of Peretz's results is true in rank $3$, but fails in general rank $n$ (assuming failure of stability of the forking relation over a base in finite rank, and the pseudolinearity conjecture), though an existential statement is true in general rank $n$. On the other hand, if there really are insurmountable errors in Peretz's proof, then then the universal implications of Peretz's statement would not be known even in rank $3$, and would also be known to be false (under the same assumptions) as in general rank $n$, but a correct version of this statement can be recovered in rank $3$ and all larger ranks by replacing it with an existential statement.} We will observe a straightforward generalization of Peretz's rank-$3$ result to higher ranks fails in an elementary way: by adding on an independent copy of an indiscernible sequence which is not a Morley sequence over a small-rank set. Still, we ask in Question \ref{term shrinking question} whether there is hope of recovering Peretz's results in higher ranks, by replacing terms with their canonical bases. It is not transparent from the proof of Peretz's result that this can be done.

The sense in which Peretz's result fails in higher ranks is in a \textit{universal} sense. That is, if the forking relation is unstable over a base (and an additional mild assumption is satisfied), it is not the case that for \textit{every} instance of instability $\{a_{i}, b_{i}\}_{i< \omega}$ of $\nind_{C}$ between types of higher rank $n$, $\{b_{i}\}_{i < \omega}$ is independent over a set of small $\mathrm{SU}$-rank. Our main lemma will be to generalize Peretz's work to higher ranks in a more unconventional way, proving an \textit{existential} result rather than a universal result. We will show that if the forking relation is unstable over a base between types of rank $n$, there \textit{exists} a base $C$, and a sequence $\{a_{i}, b_{i}\}_{i< \omega}$ exhibiting the order property for $\nind_{C}$ between types of rank $n$, such that  $\{b_{i}\}_{i < \omega}$ is independent over a small-rank set: $\{b_{i}\}_{i \geq n-1}$ is a Morley sequence over $Cb_{0} \ldots b_{n-2}$. Since Theorem \ref{main theorem} involves an existential statement, specifically the existence of an embedding of pregeometries, the statement of this paper's main theorem will be robust to our replacement of the universal quantifier with an existential one. Our proof of the existential result, like Peretz's proof of his universal result in rank $3$, will involve the independence theorem.

After giving the more powerful part of our argument by proving this existential version of Peretz's theorem, we will require some additional rank manipulations to complete the proof of our main theorem. Formally, these manipulations will generalize arguments on multi-experiment parameter identifiability corresponding to Corollary 3.2 of \cite{Ovc22}. These arguments belong to a broader body of work giving both implicit and explicit applications of canonical bases within theoretical experimental biology (\cite{Li21}, \cite{Ovc21}, \cite{Li23}, \cite{Ovc22}, \cite{ovchinnikov2025identifiable}, \cite{meshkat2025algorithm}). The original argument involves forking between entire terms within a Morley sequence over a small-rank set. However, in our refinement of this argument, the terms will be viewed as sets of realizations of an $\mathrm{SU}$-rank-$1$ partial type. We will be interested in forking between the individual realizations of that type (or in other words, the pregeoemtry on those realizations).

The organization of this paper is as follows. In Section \ref{preliminaries}, we will state the case of the stable forking conjecture that we are interested in, stability of the forking relation over a base (Conjecture \ref{stability of the forking relation over a base}). We will then review abstract pregeoemtries, and introduce the pregeometry on an $\mathrm{SU}$-rank $1$ partial type. We will conclude our setup by giving the argument for the well-known pregeometric interpretation of the stable forking conjecture, the above Fact \ref{standard pregeometric interpretation} (restated as Fact \ref{standard pregeometric interpretation, restated}). In Section \ref{an existential Peretz theorem}, we will present Peretz's original result on forking within instances of instability of $\nind_{C}$ in rank $3$ (Fact \ref{Peretz theorem}), and give the elementary argument for why this does not directly generalize to ranks higher than $3$ (Observation \ref{failure of universal Peretz theorem}). We will then give our own nonstandard generalization of Peretz's theorem, proving our existential variant of the result in higher ranks (Lemma \ref{existential Peretz theorem}). In Section \ref{a refined multi-experiment argument} we will give the final sequence modifications and rank manipulations for the proof of our main theorem, generalizing arguments from within experimental parameter identifiability (in Lemma \ref{refined Pillay-Scanlon argument}). In Section \ref{canonical base section}, we will make some additional final remarks on canonical bases, giving a rank-independent statement of our existential version of Peretz's theorem (Proposition \ref{existential Peretz theorem, rank-independent statement}) and asking a question on whether we can recover a more universal version of Peretz's work (Question \ref{term shrinking question}).

\section{Preliminaries}\label{preliminaries}

We assume a working knowledge of simplicity theory (see, e.g., \cite{Kim14} for a reference) and use standard model-theoretic notation; particularly, all sets/tuples $A, B, C, a, b, c$, etc. are assumed to be subsets of a sufficiently saturated ambient model $\mathbb{M} \models T$ of a theory $T$. We begin by describing the core basic case of the stable forking conjecture to which our results will apply; we will then review pregeometries in a minimal type, and the standard fact of geometric stability theory reducing finite-rank types to minimal types; and finally, we will state the standard, implicitly well-known pregeometric interpretation of our specialization of the stable forking conjecture, which follows from this fact of geometric stability theory.

\textbf{The stable forking conjecture:} The stable forking conjecture is stated in the introduction, in Conjecture \ref{stable forking conjecture}. Given the apparent difficulty of resolving this conjecture, we are justified in distilling the original and most general incarnation of the stable forking conjecture to a central case, which itself remains very much open. One readily apparent specialization is to relativize the stable forking conjecture to a base, or equivalently, to restrict the original stable forking conjecture to where the base is the empty set. This case is implicitly considered in Peretz (\cite{Per06}), while the main result of Brower (\cite{Brower2012}) explicitly restricts to where the base is the empty set. Moreover, relativizing to a base is obligatory for the \textit{stable Kim-forking conjecture} and \textit{simple Kim-forking conjecture} in \cite{LongKoponenConjecturePreprint}, \cite{Bos25}: these conjectures apply to all $\mathrm{NSOP}_{1}$ theories, while failing to relativize to a base results in a statement that implies simplicity, so could not apply to all $\mathrm{NSOP}_{1}$ theories.

\begin{conjecture}\label{stable forking conjecture over a base}
    \emph{(Stable forking conjecture over a base)}

    Let $T$ be a simple $\mathcal{L}$-theory. If $a \nind_{C}b$, there is a formula $\varphi(x, b) \in \mathrm{tp}(a/Cb)$ which forks over $C$, such that $\varphi(x, y)$ is a stable $\mathcal{L}(C)$-formula.
\end{conjecture}

So $\varphi(x,y)$ is only required to be stable in $\mathcal{L}(C)$, rather than in $\mathcal{L}$, which weakens the original conjecture. This is equivalent to the original stable forking conjecture for $C = \emptyset$: the based stable forking conjecture is true with the base $C$, when the original stable forking conjecture is true over the empty set in the expansion by constants for $C$.

The other main specialization is to conjecture the stability of the forking \textit{relation}, rather than the existence of a stable \textit{formula} exhibiting every instance of forking dependence. This is done independently of relativizing to a base in Section 4 of \cite{PW13}, though \cite{Per06} concentrates on the case of stability of the forking relation over a base. The underlying premise of this case is that stability applies to relations that are not necessarily definable. For a general relation to be stable, it must not exhibit the order property, just as stability of a formula means that it does not have the order property. However, when $R(x, y)$ is not necessarily a definable relation, the existence of a sequence $\{a_i, b_i\}_{i < \omega}$ with $R(a_i, a_j)$ exactly when $i > j$ does not necessarily yield an \textit{indiscernible} sequence with this property. So to define instability of a general relation, as opposed to a definable relation, we must explicitly stipulate that an instance of the order property be indiscernible. Equivalently, we may require an instance $\{a_i, b_i\}_{i < \kappa}$ of the order property for $\kappa$ large enough to apply the Erdős–Rado theorem; from this, we can extract an indiscernible instance of the order property for any (invariant) relation, while conversely, we can get a long instance of the order property from an indiscernible instance of the order property by compactness. Our main case of the stable forking conjecture will be stability of the forking relation $x\nind_{C}y$ over a base $C$.

\begin{conjecture}\label{stability of the forking relation over a base}\emph{(Stability of the forking relation over a base)}

If $T$ is simple, then the dependence relation $x \nind_{C} y$ over every base $C$ is stable: it does not have the order property, so there is no $C$-indiscernible sequence $\{a_{i}, b_{i}\}_{i < \omega}$ such that $a_{i} \nind_{C} b_{j}$ exactly when $i > j$. Equivalently, there is no sequence $\{a_{i}, b_{i}\}_{i < \kappa}$ such that $a_{i} \nind_{C} b_{j}$ exactly when $i > j$,  where $\kappa$ is large enough to apply the Erdős–Rado theorem over $C$.
\end{conjecture}

As before, stability of the forking relation over every base in every simple theory is equivalent to stability of $\nind_{\emptyset}$ in every simple theory. Stability of the forking relation over a base is also no stronger than the stable forking conjecture over a base (this is standard, but observed in Remark 7.5 of \cite{LongKoponenConjecturePreprint}): a formula exhibiting forking in the order property for $x\nind_{C} y$ would not be stable.

Because all invariant relations in finitely many variables are definable in countably categorical theories, stability of the forking relation over a base is related to the stable forking conjecture over a base in countably categorical simple theories. As observed in \cite{Per06} and the proof of Corollary 4.4 of \cite{PW13}, in a countably categorical simple theory, stability of $x \nind_{C} y$ over every finite base $C$ is equivalent to the stable forking conjecture over every finite base $C$: the relation $x \nind_{C} y$ is definable by an $\mathcal{L}(C)$-formula $\varphi(x, y)$ with parameters in $C$, for which the formula $\varphi(x, b)$ defining $x \nind_{C} b$ will fork over $C$.   Moreover, it follows from the proof of Corollary 4.4 of \cite{PW13} that, in a countably categorical \textit{supersimple} theories, stability of the forking relation over a base is equivalent to the stable forking conjecture over a base; that is, supersimplicity allows us to remove the finiteness assumption on the base. To sketch the argument, in a countably categorical simple theory satisfying stability of the forking relation over a base, if $a \nind_{C} b$, we may use supersimplicity to find finite $C_{0} \subset C$ such that $ab \ind_{C_{0}} C$. Then $a \nind_{C_{0}} b$, so $a$ satisfies $x \nind_{C_{0}} b$; by countable categoricity, $x \nind_{C_{0}} b$ will be definable by a formula $\varphi(x, b)$ that forks over $C_{0}$, which can be chosen to be an instance of the $\mathcal{L}(C_{0})$-formula $\varphi(x, y)$ defining $x \nind_{C_0} y$. The formula $\varphi(x, b)$ will even fork over $C$ because $b \ind_{C_{0}} C$, so $\varphi(x, b)$ will exhibit the dependence $a \nind_{C} b$. And $\varphi(x, y)$ will be a stable formula in $\mathcal{L}(C_{0})$, so in $\mathcal{L}(C)$, by stability of the forking relation over a base.

From the statement of the conjecture of the stability of the forking relation over a base, one can informally infer what it means for a specific theory to satisfy the conclusion of this conjecture, or to satisfy the conclusion of this conjecture between types of a specific rank. Since this will be important to the statement of the main theorem of this paper, we will nonetheless state explicitly what this means. For $T$ a finite-rank supersimple theory, we say \textit{the forking relation is stable over a base between types of rank $n$ in $T$} if for every base $C$, the forking relation $x \nind_{C} y$ over $C$ between types of rank $n$ is stable: there is no $C$-indiscernible sequence $\{a_{i}, b_{i}\}_{i < \omega}$ such that $a_{i} \nind_{C} b_{j}$ exactly when $i > j$ (or equivalently, no sequence $\{a_{i}, b_{i}\}_{i < \kappa}$ such that $a_{i} \nind_{C} b_{j}$ exactly when $i > j$,  where $\kappa$ is large enough to apply the Erdős–Rado theorem over $C$) where $\mathrm{SU}(a_{i}/C), \mathrm{SU}(b_{i}/C) \leq n$. (The inequality $\mathrm{SU}(a_{i}/C), \mathrm{SU}(b_{i}/C) \leq n$ in place of $\mathrm{SU}(a_{i}/C), \mathrm{SU}(b_{i}/C) = n$ is harmless here, giving an equivalent statement.)

Stability of the forking relation over a base, Conjecture \ref{stability of the forking relation over a base}, is clearly a different statement than the statement of the stable forking conjecture. Though the statements appear related, it does not appear easy at all to show that stability of the forking relation over a base implies the stable forking conjecture. But even under the strongest assumptions under which the full stable forking conjecture remains open, such as $\mathrm{SU}$-rank $3$, stability of the forking relation over a base also remains open. So, in choosing to concentrate on stability of the forking relation over a base, rather than the full stable forking conjecture, we have not unduly reduced the difficulty of the problem we wish to investigate.

\textbf{Minimal types, and the standard reduction to minimal types:} In supersimple theories, when one speaks of the $\mathrm{SU}$-rank $\mathrm{SU}(p)$ of a type $p$, it is often implicit that one is referring to a complete type; this is reasonable, because the $\mathrm{SU}$-rank of a partial type is defined in terms of the $\mathrm{SU}$-rank of complete types. However, the $\mathrm{SU}$-rank of a partial type is well-established; \cite{KP99} define $\mathrm{SU}$-rank for formulas as well as for complete types, using the same definition in terms of complete types as for the $\mathrm{SU}$-rank of a partial type. To our knowledge, the first explicit reference to a partial type as having $\mathrm{SU}$-rank is in \cite{Palacin2013Ample}, though the $\mathrm{SU}$-rank of a partial type is not explicitly defined there. A  definition is given in an online encyclopedia page archived by Hanson \cite{HansonWiki} (this is defined for $\mathrm{U}$-rank in stable theories, but the definition of the $\mathrm{U}$-rank is noted there to apply just as well to simple theories in the form of the $\mathrm{SU}$-rank):

\begin{definition}\label{su-rank of a partial type}

Let $p(x)$ be a partial type over a set $C$ in a simple theory. Then $\mathrm{SU}(p) := \{\mathrm{sup}(\mathrm{SU}(a/C)): a \models p(x)\}$.

\end{definition}

The case in which we will need this definition of $\mathrm{SU}$-rank for partial types, and not just complete types, will be the case of $\mathrm{SU}$-rank $1$ partial types. We can call these \textit{minimal} partial types, in keeping with the established terminology (\textit{minimal}, or \textit{weakly minimal}) for complete types. Minimal partial types, just as well as the pregeometries of strongly minimal sets defined in the foundational work of Baldwin and Lachlan mentioned in the introduction (\cite{baldwin1971strongly}), come equipped with a pregeometry given by the algebraic closure. Some precedent is found in, say, (\cite{Casanovas2006}), where it is observed that $\mathrm{SU}$-rank $1$ types as well as $D$-rank $1$ formulas (which define $\mathrm{SU}$-rank $1$ partial types) have a pregeometry given by the algebraic closure, and the class of partial types on which the algebraic closure gives a pregeometry is isolated by an explicit definition. Since we will also be interested in pregeometries in the abstract, combinatorial sense, we will first define these abstract pregeometries:

\begin{definition}\label{pregeometry definition}

A \textit{pregeometry} on a set $S$ is a map $\mathrm{cl}: \mathcal{P}(S) \to \mathcal{P}(S)$ satisfying:

\begin{itemize}
    \item (1) Reflexivity: $A \subset \mathrm{cl}(A)$

    \item (2) Monotonicity: $A \subseteq B \Rightarrow \mathrm{cl}(A) \subseteq \mathrm{cl}(B)$

    \item (3) Idempotence: $\mathrm{cl}(\mathrm{cl}(A)) = \mathrm{cl}(A)$

    \item (4) Finite character: $\mathrm{cl}(A) = \bigcup_{F \subset A \mathrm{\: finite}}\mathrm{cl}(F)$

    \item (5) Exchange: if $b \in \mathrm{cl}(Ac) \backslash \mathrm{cl}(A)$, then $c \in \mathrm{cl}(Ab) \backslash \mathrm{cl}(A)$.
\end{itemize}

\end{definition}

The pregeometry on a set also comes equipped with a \textit{dimension} function: $\mathrm{dim}(A)$ is the maximum size of a subset $A_{0} \subset A$ such that, for $a \in A_{0}$, $a \notin \mathrm{cl}(A_{0} \backslash \{a\})$. There is an alternative definition of pregeometries in terms of dimension functions, as well as a well-known cryptomorphism between pregeometries as defined in terms of dimension functions and pregeometries as defined in terms of closure operators. An \textit{embedding} of pregeometries $(S, \mathrm{cl})$ and $(S', \mathrm{cl})$ is just as the terminology might suggest: an injective map $\iota: S \hookrightarrow S' $ such that, for each $A \subset S$, $a \in S$, $a \in \mathrm{cl}(S) \Leftrightarrow \iota(a) \in \mathrm{cl}(\iota(S))$.

It is one of the foundational observations of geometric stability theory (see a standard reference on the subject, e.g., \cite{P96}) that on a strongly minimal set, or on (the set of realizations of) a minimal type in a stable theory, the algebraic closure operator $\mathrm{cl} := \mathrm{acl}$ gives a pregeometry. The argument is the same for minimal or $\mathrm{SU}$-rank $1$ partial types in a simple theory: for $p(x)$ a minimal partial type over $C$, $\mathrm{acl}_{C}$ already satisfies (1)-(4) by definition. And $\mathrm{acl}_{C}$ satisfies (5) because the relation $b \in \mathrm{acl}_{C}(Ac) \backslash \mathrm{acl}_{C}(A)$ coincides with the forking relation $b \nind_{AC} c$, for $b, c$ individual realizations and $A$ a set of realizations, and the forking relation $\nind_{AC}$ is symmetric.

\begin{definition}
    Let $p(x)$ be an $\mathrm{SU}$-rank $1$ partial type over a set $C$. Then the \textit{pregeometry} on $p(x)$ is $(S, \mathrm{cl})$, where $S$ is the set of realizations of $p(x)$ and $\mathrm{cl} :=\mathrm{acl}_{C}$ (i.e., for $A \subset S$, $\mathrm{cl}(A) = \mathrm{acl}_{C}(A) \cap S $).
\end{definition}

Because $\mathrm{SU}$-rank $1$ partial types come with a pregeometry, the standard pregeometric interpretation of the stability of the forking relation over a base will follow from reducing more general types in a finite-rank supersimple theory to types of sets of realizations of $\mathrm{SU}$-rank $1$ types. Let us begin to make this more precise, stating a standard powerful result of geometric stability theory that straightforwardly transfers to simplicity theory. We first need the following definition:

\begin{definition}
    \label{equidominance}

    Two types $p, q \in S(A)$ are \emph{equidominant} over $A$ if there are realizations $a$ and $b$ of $p$ and $q$ respectively such that, for any set $B$, $a \nind_{A} B$ if and only if $b \nind_{A} B$. Two complete types $p$, $q$ over any sets are \emph{equidominant} if there is some set $A$ containing the domains of $p$ and $q$ such that there are nonforking extensions of $p$ and $q$ to $A$ which are equidominant over $A$.
\end{definition}

The following follows from, say, Corollary 5.2.19, Proposition 5.1.12, and Lemma 5.2.11 of \cite{Wag00}. (The first of these gives the result for regular types, rather than minimal types, in supersimple theories of possibly infinite rank. The last two imply that regular types of finite $\mathrm{SU}$-rank are exactly the types of $\mathrm{SU}$-rank $1$.)

\begin{fact}
   \label{equidominancetheorem}
    
    Let $T$ be supersimple of finite rank. Then any type $p(x) \in S(A)$ is equidominant with the type, over some set $B$, of a finite $B$-independent set of realizations of minimal types over $B$.\footnote{To avoid ambiguity, let us state explicitly that we mean that these realizations of minimal types will be tuples, not singletons. So a set of realizations will really be the union of these tuples.}

\end{fact}

It follows that, for $p(x) \in S(A)$, there are $a \models p(x)$, $B \ind_{A} a$ and a $B$-independent set  $\overline{a}$ of realizations of minimal types over $B$ such that $\overline{a}$ forks with all the same sets as $a$ over $B$, and $\overline{a} \subset \mathrm{acl}(Ba)$; thus the number of independent realizations of minimal types over $B$ in $\overline{a}$, which is equal to $\mathrm{SU}(\overline{a}/B)$, is less than $\mathrm{SU}(p) = \mathrm{SU}(a/B)$. To see this, use the fact to find $a \models p(x)$, $B \ind_{A} a$ and a $B$-independent set  $\overline{a}$ of realizations of minimal types over $B$ such that $\overline{a}$ forks with all the same sets as as $a$ over $B$. Then for each realization $a_{i}$ of a minimal type over $B$ in $\overline{a}$, $a_{i} \nind_{B} \overline{a}$. So $a_{i} \nind_{B} a$. Since $\mathrm{tp}(a_{i} / B)$ is minimal, $a_{i} \in \mathrm{acl}(Ba)$. So $\overline{a} \subset \mathrm{acl}(Ba)$, because $a_{i}$ was an arbitrary member of $\overline{a}$.

\textbf{Pregeometric interpretation of stable forking dependence:} By the previous fact, we can reduce stability of the forking relation over a base, between types of rank $n$ in a finite-rank supersimple theory, to stability of the forking relation over a base between types of sets consisting of at most $n$ realizations of minimal types.

\begin{fact}\label{reduction to sets of realizations of minimal types}
    Let $T$ be a finite-rank supersimple theory. Suppose that it is not the case that the forking relation is stable over a base between types of rank $n$ in $T$. Then there is some finite set $C$, and $C$-indiscernible sequence $\{a_{i}, b_{i}\}_{i < \omega}$, such that $a_{i} \nind_{C} b_{j}$ exactly when $i > j$, and each $a_{i}, b_{j}$ is a set of size at most $n$ consisting of realizations of minimal types over $C$.
\end{fact}

\begin{proof}
    Because it is not the case that the forking relation is stable over a base between types of rank $n$ in $T$, there is some set $E$, and $E$-indiscernible sequence $\{a'_{i}, b'_{i}\}_{i < \omega}$, such that $a'_{i} \nind_{E} b'_{j}$ exactly when $i > j$, and $\mathrm{SU}(a'_{i}/E), \mathrm{SU}(b'_{i}/E) \leq n$. Then by the discussion after the previous fact, for $p(x) = \mathrm{tp}(a'_{i}/ E)$, for some $D \supset E$ there is a nonforking extension $p'(x) \in S(D)$ of $p(x)$ such that $p'(x)$ satisfies the following property:
    
    (*): for $a \models p'(x)$, there is some set $\overline{a}$, consisting of at most $n$ realizations of minimal types over $D$, which forks over $D$ with all of the same sets as $a$.

    So there is some $D \supset E$ with $D \ind_{E} a'_{0}$ such that $\mathrm{tp}(a'_{0}/ D)$ satisfies (*). By the chain condition for forking in simple theories, $D$ can be chosen such that $\{a'_{i}\}_{i < \omega}$ is indiscernible over $D$, and $D \ind_{E} \{a'_{i}\}_{i < \omega}$. Moreover, such a set $D \supset E$ can be chosen such that  $D \ind_{E \{a'_{i}\}_{i < \omega}} \{b'_{i}\}_{i < \omega}$. Therefore, $D \ind_{E} \{a'_{i}, b'_{i}\}_{i < \omega}$. Finally, such a set $D \supset E$ can additionally be chosen so that $ \{a'_{i}, b'_{i}\}_{i < \omega}$ is indiscernible over $D$. Now, because $D \ind_{E} \{a'_{i}, b'_{i}\}_{i < \omega}$, $\{a'_{i}, b'_{i}\}_{i < \omega}$ is also a $D$-indiscernible sequence exhibiting the order property for $x \ind_{D} y$: $a'_{i} \nind_{D} b'_{j}$ exactly when $i > j$. Since $\mathrm{tp}(a'_{0}/ D)$ satisfies (*), we can find some set $a_{0}$, consisting of at most $n$ realizations of minimal types over $D$, which forks over $D$ with all of the same sets as $a'_{0}$. So because $\{a'_{i}\}_{i < \omega}$ is indiscernible over $D$, for each $i < \omega$ we can find $a'_{i}$ with $a'_{i} a_{i} \equiv_{D} a'_{0}a_{0} $. Because $ \{a'_{i}, b'_{i}\}_{i < \omega}$ is indiscernible over $D$, such $a_{i}$ can additionally be chosen such that $ \{a_{i}, b'_{i}\}_{i < \omega}$ is indiscernible over $D$.  Because, for $i < \omega$, $a'_{i} a_{i} \equiv_{D} a'_{0}a_{0} $, $a_{i}$ will likewise consist of at most $n$ realizations of minimal types over $D$, and will fork over $D$ with all of the same sets as $a'_{i}$. By this last clause, and the fact that $a'_{i} \nind_{D} b'_{j}$ exactly when $i > j$, $a_{i} \nind_{D} b'_{j}$ exactly when $i > j$.

    Similarly, we may find $C \supset D$ and $b_{i}$ for $i < \omega$ such that $ \{a_{i}, b_{i}\}_{i < \omega}$ is indiscernible over $C$, $a_{i} \nind_{C} b_{j}$ exactly when $i > j$, and each $b_{i}$ is a set of size at most $n$ consisting of realizations of minimal types over $C$. But it remains the case, as over $D \subset C$, that each $a_{i}$ is a set of size at most $n$ consisting of realizations of minimal types over $C$. (Technically, if we don't assume $ C \ind_{D} \{a_{i}, b'_{i}\}_{i < \omega}$, we only get that $a_{i}$ is a set of size at most $n$ consisting of realizations of minimal types and algebraic types over $C$. But we can, of course, ignore the algebraic types over $C$.)
    
    So we have some set $C$, and $C$-indiscernible sequence $\{a_{i}, b_{i}\}_{i < \omega}$, such that $a_{i} \nind_{C} b_{j}$ exactly when $i > j$, and each $a_{i}, b_{j}$ is a set of size at most $n$ consisting of realizations of minimal types over $C$. We just need to show that we can additionally assume $C$ to be finite. By supersimplicity, there is some finite $C_{0} \subset C$ such that the two types $\mathrm{tp}(a_{i}a_{j})$ for $i > j$ and $\mathrm{tp}(a_{i}a_{j})$ for $i > j$ do not fork over $C_{0}$. Thus $a_{i} \nind_{C} b_{j}$ exactly when $i > j$, and each $a_{i}, b_{j}$ remains a set of size at most $n$ consisting of realizations of minimal types over $C_{0}$. So replacing $C$ with $C_{0}$, we can assume $C$ to be finite, as desired.
\end{proof}

From this fact, we obtain the implicitly well-known interpretation of stability for the forking relation over a base in terms of pregeometries within finite-rank supersimple theories.

\begin{fact}\label{standard pregeometric interpretation, restated}

Let $n < \omega$. Then there is a set $\mathcal{G}_{n}$ of (infinite) pregeometries, depending only on $n$, such that the following are equivalent for any finite-rank supersimple theory $T$:

    \begin{itemize}
        \item In $T$, the forking relation is stable over a base between types of rank $n$.
        \item No pregeometry in $\mathcal{G}_{n}$ embeds into the pregeometry on an $\mathrm{SU}$-rank-$1$ partial type over a finite set in $T$.
    \end{itemize}
    
\end{fact}

\begin{proof}
Let $\kappa$ be large enough to apply the Erdős–Rado theorem over a finite set. Then define $\mathcal{G}_{n}$ to be the set consisting of pregeometries $(S, \mathrm{cl})$, where $S$ can be written as a union $\bigcup_{i < \kappa} (a_{i} \cup b_{i})$ for sets $a_{i}, b_{i}$ of size at most $n$, such that $\mathrm{dim}(a_{i} \cup a_{j}) < \mathrm{dim}(a_{i}) + \mathrm{dim}(a_{j})$ exactly when $i > j$.

We first show that the first bullet point implies the second. Suppose that the pregeometry on some minimal partial type $p(x)$ over a finite set $C$ embeds a pregeometry in $\mathcal{G}_{n}$. Define the sequence $\{a_{i}, b_{i}\}_{i < \kappa}$ so that, by abuse of notation, $a_{i}, b_{i}$ is the image of $a_{i}, b_{i}$ under the embedding. Then $a_{i} \nind_{C} b_{j}$ exactly when $i > j$, and $a_{i}, b_{i}$ have $\mathrm{SU}$-rank at most $n$ over $C$. So the forking relation is unstable over a base between types of rank $n$, as desired.

It remains to show that the second bullet point implies the first. Suppose the forking relation is unstable over a base between types of rank $n$. Then by the previous fact, for some finite set $C$ there is some $C$-indiscernible sequence $\{a_{i}, b_{i}\}_{i < \omega}$, such that $a_{i} \nind_{C} b_{j}$ exactly when $i > j$, and each $a_{i}, b_{j}$ is a set of size at most $n$ consisting of realizations of minimal types over $C$. We may extend $\{a_{i}, b_{i}\}_{i < \omega}$ to a $C$-indiscernible sequence $\{a_{i}, b_{i}\}_{i < \kappa}$ satisfying the same conditions. Then there are $\mathrm{SU}$-rank-$1$ types $p_{1}(x), \ldots p_{k}(x) \in S(C)$ for $k \leq 2n$ such that for each $i < \kappa$, each of the realizations of $\mathrm{SU}$-rank-$1$ types over $C$ within the sets of realizations $a_{i}, b_{i}$ satisfies one of the $p_{\ell}(x)$. Recall that the disjunction of finitely many types, and particularly complete types, is a partial type. Let $p(x)$ be the disjunction of all of the  $p_{1}(x), \ldots p_{k}(x)$ (i.e., the type defining the union of the sets of realizations of $p_{1}(x), \ldots p_{k}(x)$). Then $p(x)$ has $\mathrm{SU}$-rank $1$, because each realization satisfies one of the $\mathrm{SU}$-rank-$1$  complete types $p_{1}(x), \ldots p_{k}(x)$. Moreover, each of the realizations within each $a_{i}, b_{i}$ satisfies $p(x)$. So the sequence $\{a_{i}, b_{i}\}_{i < \kappa}$ gives an embedding from a member of $\mathcal{G}_{n}$ into the pregeometry on $p(x)$, as desired. (For the injectivity, note that none of the realizations in $a_{i}$ will be equal to one of the realizations in one of the $b_{j}$; otherwise $a_{i} \nind_{C} b_{j}$ for all $i, j$.)
\end{proof}

Of course, as this proof shows, the $\mathrm{SU}$-rank-$1$ partial types we are talking about in this fact are all just finite disjunctions of complete types of $\mathrm{SU}$-rank $1$; the same will be true for the main theorem of this paper, where we show $\mathcal{G}_n$ can be taken to be \textit{finite}.

\section{An existential Peretz theorem}\label{an existential Peretz theorem}

The goal of the next two sections will be to prove the main theorem of this paper, as described in the introduction. Again, this says that, within the well-known pregeometric interpretation of stability of the forking relation over a based between types of rank $n$ (Fact \ref{standard pregeometric interpretation, restated}), this stable forking hypothesis actually reduces to a \textit{finite} collection $\mathcal{G}_{n}$ of pregeometries.

\begin{theorem}\label{main theorem, restated}

Let $n < \omega$. Then there is a \emph{finite} set $\mathcal{G}_{n}$ of (infinite) pregeometries, depending only on $n$, such that the following are equivalent for any finite-rank supersimple theory $T$:

    \begin{itemize}
        \item In $T$, the forking relation is stable over a base between types of rank $n$.
        \item No pregeometry in $\mathcal{G}_{n}$ embeds into the pregeometry on an $\mathrm{SU}$-rank-$1$ partial type over a finite set in $T$.
    \end{itemize}

\end{theorem}

Throughout the rest of this section and paper, we will assume that any theory we are referring to is a finite-rank supersimple theory.

The goal of this section will be to present the more theoretically powerful part of the argument for this theorem. We will start with a $C$-indiscernible sequence $\{a_i, b_i\}_{i < \omega}$ giving an instance of the order property for the forking relation $x\nind_{C} y$ over a base $C$. From this, we will obtain a bound on the size of the initial segment of, say, $\{ b_i\}_{i < \omega}$  over which we can assume the limit type of $\{ b_i\}_{i < \omega}$ does not fork (relative to the base $C$). The bound obtained in this section will later extend to the limit type of $\{a_i, b_i\}_{i < \omega}$ over $C$; when the $a_{i}, b_{i}$ are each sets of individual realizations of an $\mathrm{SU}$-rank-$1$ partial type $p(x)$ over $C$, the arguments of the next section will use this bound to describe the forking between the realizations of $p(x)$ within $\{a_i, b_i\}_{i < \omega}$. Our description of forking between these realizations will be full enough that we may produce pregeometries within the original $\mathcal{G}_{n}$ from one of finitely many possibilities.

The least-rank open case of stability of the forking relation over a base (even assuming countable categoricity) is between types of rank $3$. In rank $3$, our core segment-bounding argument is already implicitly proven by Peretz in \cite{Per06}. So the case $n = 3$ of our main theorem is really due to Peretz, up to an especially low-rank case of the refined multi-experiment argument from the next section. We will start by stating Peretz's theorem, and extracting the segment-bounding result in rank $3$ that we would like to generalize to higher ranks. In higher ranks, the result breaks down in a way Peretz may not have foreseen: we can only prove an \textit{existence} statement for an instance $\{a_i, b_i\}_{i < \omega}$ of the order property for $x\nind_{C} y$ (under the obligatory background assumption that the forking relation \textit{is} unstable over a base between types of rank $n$), rather than a universal statement as in rank $3$. Fortuitously for us, despite the divergence in higher ranks from Peretz's theorem, the main theorem we wish to prove is exactly the kind that can tolerate this leap from the universal into the existential.

Let $\{a_i, b_i\}_{i < \omega}$ be a $C$-indiscernible sequence exhibiting the instability of $x \nind_{C} y$ between types of rank $3$, so $a_{i}, b_{i}$ have $\mathrm{SU}$-rank $3$ over $C$ and $a_{i} \nind_{C} b_{j}$ exactly when $i > j$. In our language (and following Peretz's observation that his proofs work for the forking relation over a base in general, outside of his chosen setting of countably categorical theories), Peretz proves the following. The sequence $\{b_i\}_{i < \omega}$, in the terminology of Peretz, will be \textit{$3$-$n$} for $n =1, 2$ if $\mathrm{SU}(b_{1}/Cb_{0})= n$, and \textit{$3$-$3$-$n$} for $n =1, 2$ if $b_{1} \ind_{C} b_{0}$, $\mathrm{SU}(b_{2}/Cb_{0}b_{1})= n$. Peretz notes that the four \textit{prima facie} possibilities for $\{b_i\}_{i < \omega}$ are that it is $3$-$2$, $3$-$1$, $3$-$3$-$2$ and $3$-$3$-$1$.

Peretz's theorem, Theorem 3 of \cite{Per06}, states that among these possibilities, $\{b_i\}_{i < \omega}$ must be $3$-$3$-$1$. Like his argument for stability of the forking relation over a base between types of rank $2$, which is later supplanted by Brower (\cite{Brower2012}), Peretz directly applies the independence theorem in the proof of this theorem. This differs from arguments elsewhere in the existing literature on the stable forking conjecture. The central consequence of Peretz's result, at least for our purposes, give a bound on the initial segment over which the limit type of $\{b_i\}_{i < \omega}$ does not fork (again, relative to the base $C$). Recall that, for $I = \{c_{i}\}_{i < \omega}$ an $A$-indiscernible sequence, the \textit{limit type} $\mathrm{lim}_{A}(I^{+}/I)$ of $\{c_{i}\}_{i < \omega}$ over $A$ is just $\mathrm{tp}(c_{\omega}/A \{c_{i}\}_{i < \omega} )$, where $\{c_{i}\}_{i < \omega}$ is extended to an $A$-indiscernible sequence $\{c_{i}\}_{i \leq \omega}$. If the $C$-indiscernible sequence $I:=\{b_i\}_{i < \omega}$ must be $3$-$3$-$1$, then $\mathrm{SU}(\mathrm{lim}_{C}(I^{+}/I))=1$: $\mathrm{SU}(\mathrm{lim}_{A}(I^{+}/I))\neq 0$ because otherwise $\{b_i\}_{i < \omega}$ would be constant, contradicting that $\{a_i, b_i\}_{i < \omega}$ exhibits the instability of $x \nind_{C} y$. But the restriction of the limit type of $\{b_i\}_{i < \omega}$ over $C$ to $Cb_{0}b_{1}$ is $\mathrm{tp}(b_{2}/Cb_{0}b_{1})$, so has $\mathrm{SU}$-rank $1$. Thus the limit type of $\{b_i\}_{i < \omega}$ over $C$ does not fork over $b_{0}b_{1}$. In conclusion, Peretz shows the following fact:

\begin{fact} (Peretz, \cite{Per06})
\label{Peretz theorem}

Let $\{a_i, b_i\}_{i < \omega}$ be a $C$-indiscernible sequence exhibiting the instability of $x \nind_{C} y$ between types of rank $3$, so $a_{i}, b_{i}$ have $\mathrm{SU}$-rank $3$ over $C$ and $a_{i} \nind_{C} b_{j}$ exactly when $i > j$. Then the limit type of $\{b_i\}_{i < \omega}$ over $C$ does not fork over $Cb_{0}b_{1}$, so $\{b_{i}\}_{i \geq 2}$ is a Morley sequence over $Cb_{0}b_{1}$.

\end{fact}

As mentioned in the introduction, Peretz suggests generalizations of this result to higher ranks. But under a mild assumption, the universal result of Fact \ref{Peretz theorem}, which applies to \textit{every} sequence $\{a_i, b_i\}_{i < \omega}$ exhibiting instability of the forking relation over a base, will fail in higher ranks. The reason why the universal result will fail under this mild assumption is elementary: our assumption will yield an indiscernible sequence of rank-$2$ parameters whose limit type forks over an arbitrarily large initial segment, and we just add on the terms of an independently chosen copy of this sequence to the terms of the sequence $\{a_i, b_i\}_{i < \omega}$ exhibiting instability of the forking relation. We first detail the mild assumption.

Let $p(x) \in S(C)$ have $\mathrm{SU}$-rank $1$. The type $p(x)$ is \textit{k-linear} if $k$ is the maximum of the ranks $\mathrm{SU}(\mathrm{Cb}(q)/C)$ for $q = \mathrm{tp}(ab/A)$, $A \supset C$, $a, b \models p(x)$. (See, say, \cite{P96}, following \cite{Bu91}, for this definition.) As stated in, say, the introduction to \cite{BTW02}, every $1$-linear (or \textit{linear}) minimal type is one-based. Moreover, Hart, Kim and Pillay (\cite{HKP00}, Proposition 4.6) prove that if all minimal types are one-based, a theory must be ``modular," equivalently one-based. Kim's \textit{pseudolinearity conjecture} (\cite{Kim10}, after Buechler's theorem \cite{Bu91}) posits that every minimal type is either linear, or not even $k$-linear for any $k$. Tomašić and Wagner (\cite{TW03}, using Buechler's group configuration theorem technique \cite{BTW02} and the group configuration theorem for countably categorical simple theories from  \cite{Tomasic01}, \cite{BenYaacov2004}) prove the conclusion of the pseudolinearity conjecture for finite-rank supersimple theories that are countably categorical.

Assume the conclusion of the pseudolinearity conjecture, and let $\{a_i, b_i\}_{i < \omega}$ be a $C$-indiscernible sequence exhibiting the instability of $x \nind_{C} y$ between types of rank $n$, so $a_{i}, b_{i}$ have $\mathrm{SU}$-rank $n$ over $C$ and $a_{i} \nind_{C} b_{j}$ exactly when $i > j$. We know that the $\mathcal{L}$-theory is not one-based in $\mathcal{L}(C)$: supersimple theories eliminate hyperimaginaries (\cite{BuechlerPillayWagner2001}), and Kim proves the conclusion of the stable forking conjecture for one-based theories eliminating hyperimaginaries (Theorem 4.3 of \cite{Kim01}). So by the previous paragraph, there is some $C' \supset C$ with some minimal type $p(x) \in S(C')$ that is not $k$-linear for any $k$. We may choose $C' \ind_{C} \{a_i, b_i\}_{i < \omega}$ with $\{a_i, b_i\}_{i < \omega}$ $C'$-indiscernible, so it remains the case that $a_{i}, b_{i}$ have $\mathrm{SU}$-rank $n$ over $C'$ and $a_{i} \nind_{C'} b_{j}$ exactly when $i > j$; thus it does no harm to assume that there is some complete minimal type $p(x)$, not $k$-linear for any $k$, over $C$ itself.

We apply the indiscernible sequence interpretation of $k$-linearity in \cite{LongKoponenConjecturePreprint}: that $p(x)$ is not $k$-linear for any $k$ implies that for all $k$ there is $A \supset C$, $q(y) \in S(A)$, $\mathrm{SU}(q(y)|_{C})= 2$ with $\mathrm{SU}(\mathrm{Cb}(q)/ C) \geq k$. Let $\{c_{i}\}_{i < \omega}$ by a Morley sequence over $A$ consisting of realizations of $q(y)$; then $\{c_{i}\}_{i < \omega }$ will not be a Morley sequence over $C$, because otherwise $A \ind_{C} c_{0}$ by Kim's lemma, so by symmetry $q(y) = \mathrm{tp}(c_{0}/A)$ would not fork over $C$, contradicting $\mathrm{SU}(\mathrm{Cb}(q)/ C) \geq k > 0$. For $k > 2$, $\mathrm{SU}(\mathrm{Cb}(q)/ C) \geq k$ will imply that $\mathrm{SU}(q) = 1$ ($q$ forks over $C$ with $\mathrm{SU}(q(y)|_{C})= 2$, so $\mathrm{SU}(q) \leq 1$; $\mathrm{Cb}(q)$ is not contained in $\mathrm{acl}(aC)$ for $a \models q(y)$, so $\mathrm{SU}(q) > 0$.) So $\mathrm{SU}(q) - \mathrm{SU}(q|_{C}) = 1$ with $\{c_{i}\}_{i < \omega }$ a Morley sequence over $A$ consisting of realizations of $q(y)$ and $\mathrm{SU}(\mathrm{Cb}(q)/ C) \geq k$. Fact 3.2.3.2 of \cite{LongKoponenConjecturePreprint} says exactly that under these hypotheses, any $k$ terms of $\{c_{i}\}_{i < \omega }$ are independent (over $C$). So we can conclude that, for arbitrarily large $k$, there is an indiscernible sequence over $C$, whose terms have $\mathrm{SU}$-rank $2$ over $C$, that is not a Morley sequence over $C$, but for which any $k$ terms are independent. (We could have just quoted Corollary 3.2.9 of \cite{LongKoponenConjecturePreprint} for this, but we spelled out some of the intermediate arguments to make the construction more explicit.) Thus, for arbitrarily large $k$, there is a $C$-indiscernible sequence $\{c_{i}\}_{i < \omega}$ with $\mathrm{SU}(c_{i}/ C) = 2$ whose limit type over $C$ forks over $Cc_{0}\ldots c_{k-2}$. (If $\{c_{i}\}_{i < \omega}$ is not a Morley sequence but any $k$ terms are independent, then the limit type over $C$ must fork over $C$. But the restriction of the limit type to  $c_{0}\ldots c_{k-2}$ does not fork over $C$ because any $k$ terms are independent. So by transitivity, the limit type over $C$ forks over $Cc_{0}\ldots c_{k-2}$.)

We may choose $\{c_{i}\}_{i < \omega}$ as above such that $\{c_{i}\}_{i < \omega} \ind_{C} \{a_i, b_i\}_{i < \omega}$. We may additionally choose $\{c_{i}\}_{i < \omega}$ such that, for $\tilde{b}_{i} : = b_{i}c_{i}$, $\{a_i, \tilde{b}_i\}_{i < \omega}$ remains indiscernible over $C$. Then $\mathrm{SU}(\tilde{b}_{i}/C) = n+ 2 $, and $\{a_i, \tilde{b}_i\}_{i < \omega}$ is a $C$-indiscernible sequence exhibiting instability of the forking relation $x \nind_{C} \tilde{y}$ between types of rank at most $n + 2$: $a_{i} \nind_{C} \tilde{b}_{j}$ exactly when $i > j$. Moreover, the limit type of $\{ \tilde{b}_i\}_{i < \omega}$ over $C$ forks over $C\tilde{b}_{0} \ldots \tilde{b}_{k-2}$, because the limit type of $\{c_{i}\}_{i < \omega}$ over $C$ forks over $Cc_{0}\ldots c_{k-2}$. Since $k$ was arbitrarily large, we see that under our assumptions, Peretz's theorem Fact \ref{Peretz theorem} has no universal analogue in higher ranks, barring the vacuous case where stability of the forking relation over a base is just true. More precisely:

\begin{observation}\label{failure of universal Peretz theorem}
Suppose that, in a finite-rank supersimple theory $T$ satisfying the conclusion of the pseudolinearity conjecture, stability of the forking relation over a base fails. Then there is $n < \omega$ such that, for some $C$ and arbitrarily large $k < \omega$, there is a $C$-indiscernible sequence $\{a_i, b_i\}_{i < \omega}$ exhibiting the instability of $x \nind_{C} y$ between types of rank $n$ (so $a_{i}, b_{i}$ have $\mathrm{SU}$-rank $n$ over $C$ and $a_{i} \nind_{C} b_{j}$ exactly when $i > j$), with the limit type of $\{b_{i}\}_{i < \omega}$ over $C$ forking over the initial segment $Cb_{0}\ldots b_{k-1}$.

\end{observation}

The reason why this observation poses no issue for the case $n =3$ is that, according to the argument for the observation, if stability of the forking relation fails over a base between types of rank $n$, Peretz's theorem must have no universal analogue in the rank $n + 2$.  So hypothetically, if we were to attempt to apply the argument to obtain a failure of Peretz's theorem in his original rank-$3$ case, or even the rank-$4$ case, we would need to start with an instance of instability of the forking relation over a base between types of rank $1$ or rank $2$. This assumption would be vacuous, as by Brower's theorem (Theorem 2.3.4 of \cite{Brower2012}), the stable forking conjecture over a base is true in rank $2$. But by the same token, if stability of the forking relation over a base fails in the lowest-rank case where it is open, between types of rank $3$, Peretz's theorem will have no universal analogue between a type of rank $3$ and a type of rank $5$.

In Question \ref{term shrinking question} below, in Section \ref{canonical base section}, we will propose a way in which Peretz's universal result in higher ranks might be recovered, just by shrinking the terms $a_{i}, b_{i}$ within an instance $\{a_i, b_i\}_{i < \omega}$ of $\nind_{C}$. The argument that Peretz's universal result fails in higher ranks suggests this possibility: suppose $\{a_i, b_i\}_{i < \omega}$ exhibits instability of $\nind_{C}$, and the limit type of $\{ b_i\}_{i < \omega}$ does not fork over the initial segment $C b_{0} \ldots b_{N-1}$. Then if we add the terms of $\{c_{i}\}_{i <\omega }$ to get $\{a_i, \tilde{b}_i\}_{i < \omega}$ above, this construction does not preclude shrinking the terms $\tilde{b}_i$ back to $b_{i}$ to recover the sequence $\{a_i, b_i\}_{i < \omega}$ satisfying the segment bound of size $N$. However, it remains open, and not readily apparent from the proof of Peretz's theorem (Theorem 3, \cite{Per06}), whether we can actually recover Peretz's universal result in higher ranks in this way.

The main lemma of this section will say that, if stability of the forking relation fails over a base between types of rank $n$, then there will be \textit{some} base $C$ and rank-$n$ instance of instability for $\nind_{C}$ that satisfies Peretz's segment-bounding conclusion. This is in contrast to our observation that Peretz's segment-bounding conclusion, when the forking-stability hypothesis fails under the conclusion of the pseudolinearity conjecture, does not hold for \textit{every} base $C$ and rank-$n$ instance of instability for $\nind_{C}$.

\begin{lemma}\label{existential Peretz theorem}

Suppose that it is not the case that the forking relation is stable over a base between types of rank $n$. Then there \emph{exists} a finite set $C$ and a $C$-indiscernible sequence $\{a_i, b_i\}_{i < \omega}$ exhibiting the instability of $x \nind_{C} y$ between types of rank $n$ (so $a_{i}, b_{i}$ have $\mathrm{SU}$-rank (at most) $n$ over $C$, and $a_{i} \nind_{C} b_{j}$ exactly when $i > j$) such that the limit type of $\{b_i\}_{i < \omega}$ over $C$ does not fork over $Cb_{0} \ldots b_{n-2}$, so $\{b_{i}\}_{i \geq n-1}$ is a Morley sequence over $Cb_{0} \ldots b_{n-2}$.

\end{lemma}

Assuming we have some instance of instability of the forking relation over a base between types of rank $n$ to start with, our proof will \textit{not} suggest that the relationship between $C$,  $\{a_i, b_i\}_{i < \omega}$ as in the lemma, and the original instance of instability of the forking relation over a base, will just involve adding and dropping terms. (Question \ref{term shrinking question} will propose this possibility, asking whether the relationship between the two sequences just involves shrinking terms.) Instead, the proof will suggest that finding the sequence in the lemma will require materially altering another sequence via applications of the independence theorem (Case 2 in the proof). Our arguments will have a similar flavor to some of Peretz's arguments for Fact \ref{Peretz theorem}, Theorem 3 in \cite{Per06}. As noted earlier, there is precedent from Peretz's results in \cite{Per06} for applying the independence theorem to types of terms of a sequence exhibiting instability of the forking relation over a base. We note in particular the resemblance of point (1) of the proof of Claim 3.2 of \cite{Per06} to Case 1 below, and the larger argument in points (0)-(3) there to case $2$.

We prove Lemma \ref{existential Peretz theorem}:

\begin{proof}
  We are assuming that there is some base $C$ and a $C$-indiscernible sequence $\{a_i, b_i\}_{i < \omega}$ exhibiting the instability of $x \nind_{C} y$ between types of rank $n$: $a_{i}, b_{i}$ have $\mathrm{SU}$-rank (at most) $n$ over $C$, and $a_{i} \nind_{C} b_{j}$ exactly when $i > j$. (We may assume $C$ is finite, as in the last paragraph of the proof of Fact \ref{reduction to sets of realizations of minimal types}.) For a $C$-indiscernible sequence $\{b_{i}\}_{i < \omega}$ with $\mathrm{SU}(b_{i}/ C) \leq n$, the following are equivalent:

  \begin{itemize}
      \item The sequence $\{b_{i}\}_{i < \omega}$ admits a sequence $\{a_{i}\}_{i < \omega}$ with $\{a_i, b_i\}_{i < \omega}$ exhibiting the instability of $x \nind_{C} y$ between types of rank $n$.

      \item There is $a$ with $\mathrm{SU}(a/C) \leq n$, and a $C$-indiscernible sequence $\{ b_{i}\}_{i < \omega+\omega}$ extending $\{ b_{i}\}_{i < \omega}$, such that $\{ b_{i}\}_{i < \omega}$ is $Ca \{b_{i}\}_{\omega \leq i < \omega+\omega }$-indiscernible, $\{ b_{i}\}_{\omega \leq i < \omega+\omega}$ is $Ca \{b_{i}\}_{i < \omega }$-indiscernible, $a \nind_{C} b_{0}$, and $a \ind_{C} b_{\omega}$.
  \end{itemize}
  
   (The proof of Theorem 4.1 of \cite{PW13} observes the direction from $\{a_i, b_i\}_{i < \omega}$ to $a$, $\{b_{i}\}_{i < \omega + \omega}$.)
  
  Among all $C$, $\{b_{i}\}_{ i < \omega+\omega }$, $a$ satisfying these above conditions ($a$, $b_{i}$ have $\mathrm{SU}$-rank at most $n$ over $C$, $\{ b_{i}\}_{i < \omega+\omega}$ is $C$-indiscernible, $\{ b_{i}\}_{i < \omega}$ is $Ca \{b_{i}\}_{\omega \leq i < \omega+\omega }$-indiscernible, $\{ b_{i}\}_{\omega \leq i < \omega+\omega}$ is $Ca \{b_{i}\}_{i < \omega }$-indiscernible, $a \nind_{C} b_{0}$, $a \ind_{C} b_{\omega}$), let us choose $C$, $\{b_{i}\}_{\omega \leq i < \omega+\omega }$, $a$ to make the following minimizations and maximizations.

  (1) The rank $\mathrm{SU}(a/C)$ is minimized.

  (2) Subject to (1), for $I = \{b_{i}\}_{ i < \omega+\omega } $ and $\mathrm{lim}_{C}(I^{+}/I)$ its limit type over $C$ (i.e., $\mathrm{tp}(b_{\omega+\omega}/C \{b_{i}\}_{ i < \omega+\omega })$ for $\{b_{i}\}_{i \leq \omega+\omega }$ a $C$-indiscernible sequence extending $\{b_{i}\}_{ i < \omega+\omega }$), $\mathrm{SU}(\mathrm{lim}_{C}(I^{+}/I))$ is maximized. (Note that in any case,  $\mathrm{SU}(\mathrm{lim}_{C}(I^{+}/I)) \leq \mathrm{SU}(b_{\omega+\omega}/ C) \leq n$, so it makes sense to say there is a maximum.)

  By the equivalence between the two bullet points, it suffices to show that the limit type of $\{b_i\}_{i < \omega + \omega}$ does not fork over $Cb_{0} \ldots b_{n-2}$ for this choice of $C$, $\{b_{i}\}_{ i < \omega+\omega }$, $a$. Since this limit type is finitely satisfiable in $C\{b_i\}_{i < \omega}$, it suffices to show that the limit type of $\{b_i\}_{i < \omega}$ does not fork over $Cb_{0} \ldots b_{n-2}$, equivalently $b_{\omega} \ind_{Cb_{0} \ldots b_{n-2}} \{b_i\}_{i < \omega}$. Let $m$ be the $\mathrm{SU}$-rank of the limit type of $\{b_i\}_{i < \omega}$ over $C$, $\mathrm{tp}(b_{\omega}/C \{b_{i}\}_{ i < \omega})$. Suppose for a contradiction that  $b_{\omega} \nind_{Cb_{0} \ldots b_{n-2}} \{b_i\}_{i < \omega}$; then $\mathrm{SU}(\mathrm{tp}(b_{\omega}/C b_{0}\ldots b_{n-2})) > \mathrm{SU}(\mathrm{tp}(b_{\omega}/C \{b_{i}\}_{ i < \omega})) = m$.

 We know that it is not the case that, for each of the $n$ values $i$ with $0 \leq i \leq n-1$, $a \nind_{Cb_{0} \ldots b_{i-1}} b_{i}$ (where the expression $b_{0} \ldots b_{i}$ is empty for $i = 0$.) Otherwise, $\mathrm{SU}(a/b_{0} \ldots b_{n-1})= 0$, so $a \in \mathrm{acl}_{C}(b_{0}, \ldots b_{n-1})$. By, say, $a \nind_{C} b_{0}$ and Kim's lemma, $\{b_{i}\}_{i < \omega}$ cannot be a Morley sequence. So $b_{\omega} \nind_{C} \{b_{i}\}_{i < \omega}$, so $b_{\omega} \nind_{C} \mathrm{acl}_{C}(\{b_{i}\}_{i < \omega})$, so $b_{\omega} \nind_{C} a$ and $a \nind_{C} b_{\omega} $ by symmetry, a contradiction.

 So for some $k$ with $0 \leq k \leq n-1$, $a \ind_{Cb_{0} \ldots b_{k-1}} b_{k}$. Moreover, $a \nind_{C} b_{0}$ implies $k > 0$. There are two cases, $a \nind_{Cb_{0} \ldots b_{k-1}} b_{\omega}$ and $a \ind_{Cb_{0} \ldots b_{k-1}} b_{\omega}$.

 \textit{Case 1: $a \nind_{Cb_{0} \ldots b_{k-1}} b_{\omega}$.} Then for the base $D := Cb_{0} \ldots b_{k-1}$, $a \nind_{C} b_{0}$, $k > 0$ implies $\mathrm{SU}(a/D) < \mathrm{SU}(a/C)$. Then $D$, $\{b_{i}\}_{k\leq i \leq i < \omega+\omega }$, $a$ satisfy the following: $a$, $b_{i}$ have $\mathrm{SU}$-rank at most $n$ over $D$, $\{ b_{i}\}_{k \leq i < \omega+\omega}$ is $D$-indiscernible $\{ b_{i}\}_{k \leq i < \omega}$ is $Da \{b_{i}\}_{\omega \leq i < \omega+\omega }$-indiscernible, $\{ b_{i}\}_{\omega \leq i < \omega+\omega}$ is $Da \{b_{i}\}_{k \leq i < \omega }$-indiscernible, $a \ind_{D} b_{k}$, $a \nind_{D} b_{\omega}$. Flipping, we get a contradiction to the minimality in (1).

 \textit{Case 2: $a \ind_{Cb_{0} \ldots b_{k-1}} b_{\omega}$.} Similarly to before, define $D := \mathrm{acl}^{\mathrm{eq}}(Cb_{0} \ldots b_{k-1})$ (though we will not end up using $D$ as the new base for instability of $\nind_{D}$, as in Case 1). The purpose of $\mathrm{acl}^{\mathrm{eq}} = \mathrm{bdd}$ (\cite{BuechlerPillayWagner2001}) will be to apply the independence theorem over $D$. Then $b_{k} \equiv_{D} b_{\omega}$ (because $b_{k}$, $b_{\omega}$ belong to an indiscernible sequence over $Cb_{0} \ldots b_{k-1}$, so over $D=\mathrm{acl}^{\mathrm{eq}}(Cb_{0} \ldots b_{k-1})$), $a \ind_{D} b_{k}$, and $a \ind_{D} b_{\omega}$. Moreover, $\mathrm{SU}(b_{k}/D) = \mathrm{SU}(b_{\omega}/D)= \mathrm{SU}(\mathrm{tp}(b_{\omega}/C b_{0}\ldots b_{k-1})) \geq \mathrm{SU}(\mathrm{tp}(b_{\omega}/C b_{0}\ldots b_{n-2})) > m$. For $p(x) := \mathrm{tp}(b_{k}/D)=\mathrm{tp}(b_{\omega}/D)$, let $\{b'_{i}\}_{i < \omega}$ be a Morley sequence over $D$ consisting of realizations of $p(x)$. The $\mathrm{SU}$-rank of the limit type of $\{b'_{i}\}_{i < \omega+\omega}$ over $C$ is equal to $\mathrm{SU}(p(x)) > m$, by the standard argument:  the limit type of $\{b'_{i}\}_{i < \omega+\omega}$ over $D$ does not fork over $D$ (it is a Morley sequence over $D$), nor over $C\{b'_{i}\}_{i < \omega+\omega}$ (by finite satisfiability in $C\{b'_{i}\}_{i < \omega+\omega}$). So if we can find $a' \equiv_{C} a$ such that $\{ b'_{i}\}_{i < \omega}$ is $Ca' \{b'_{i}\}_{\omega \leq i < \omega+\omega }$-indiscernible, $\{ b'_{i}\}_{\omega \leq i < \omega+\omega}$ is $Ca' \{b'_{i}\}_{i < \omega }$-indiscernible, $a' \nind_{C} b'_{0}$, $a' \ind_{C} b'_{\omega}$, we will have reached a contradiction to the maximality in (2). By $a \ind_{D} b_{k}$, $a \ind_{D} b_{\omega}$, and the fact $\{ b'_{i}\}_{i < \omega + \omega}$ is a Morley sequence of $D$, we may repeatedly apply the independence theorem over $D$, obtaining $a \ind_{D} \{ b'_{i}\}_{i < \omega + \omega}$ such that $a'b'_{i} \equiv_{D} a b_{k} $ for $i < \omega$, $a'b'_{i} \equiv_{D} a b_{\omega} $ for $\omega \leq i < \omega+\omega$. Particularly, $a'b'_{i} \equiv_{C} a b_{k} $ for $i < \omega$, $a'b'_{i} \equiv_{C} a b_{\omega} $ for $\omega \leq i < \omega+\omega$. We may choose $a'$ subject to the last sentence such that additionally, $\{ b'_{i}\}_{i < \omega}$ is $Ca' \{b'_{i}\}_{\omega \leq i < \omega+\omega }$-indiscernible and $\{ b'_{i}\}_{\omega \leq i < \omega+\omega}$ is $Ca' \{b'_{i}\}_{i < \omega }$-indiscernible. (As in the proof of Lemma 1.2 of \cite{Che14}, we may first choose such $a'$ such that the first condition is satisfied. Then, given that we have such $a'$ such that the first indiscernibility condition is satisfied, we can choose such $a'$ such that both indiscernibility conditions are satisfied.) Then by $a'b'_{0} \equiv_{C} a b_{k} $ and $a \nind_{C} b_{k} $, $a' \nind_{C} b'_{0}$. Similarly, by $a'b'_{\omega} \equiv_{C} a b_{\omega} $ and $a \ind_{C} b_{\omega} $, $a' \ind_{C} b'_{\omega}$. This gives us the desired contradiction.

\end{proof}

Note that the proof of this lemma says slightly more: suppose $x\nind_{C'}y$ is unstable between types $p(x), q(y) \in S(C')$ of $\mathrm{SU}$-rank $n$, so there is a $C'$-indiscernible sequence $\{a'_i, b'_i\}_{i< \omega}$ with $a'_{i} \nind_{C'} b'_{j}$ exactly when $i > j$ and $a'_{i} \models p(x)$, $b'_{i} \models q(y)$. Then there is $C$, $\{a_i, b_i\}_{i< \omega}$ as in the statement of the lemma with $C \supset C'$, $a_{i} \models p(x)$, $b_{i} \models q(y)$. We will use this statement in the next section to give the proof of our main theorem, Theorem \ref{main theorem, restated}.

\section{A refined multi-experiment argument}\label{a refined multi-experiment argument}

We have seen that, if the forking relation is unstable in a finite-rank supersimple theory between types of rank $n$, then \textit{there is} an instance of this forking instability satisfying 
Peretz's nonforking condition over a bounded initial segment. This is notwithstanding the fact that, unlike in Peretz's result in rank $3$, not \textit{every} instance of instability in rank $n$ satisfies this segment-bounding condition. Earlier, we also reviewed the well-known fact that instability of the forking relation between types of rank $n$ is equivalent to some minimal type pregeometry embedding a pregeometry from a specific set $\mathcal{G}_{n}$ (Fact \ref{standard pregeometric interpretation, restated}). And our goal in proving our main theorem (Theorem \ref{main theorem, restated}) is to show that $\mathcal{G}_{n}$ can be taken to be finite. We will now see that we can obtain not just any pregeometry in this $\mathcal{G}_{n}$, but in fact one of only finitely many specific pregeometries in $\mathcal{G}_{n}$, from an instance of instability of $x\nind_{C}y$ satisfying Peretz's condition. So, because we know that instability of $x\nind_{C}y$ between types of rank $n$ implies the existence of an instance of instability of $x\nind_{C}y$ satisfying Peretz's condition (Lemma \ref{existential Peretz theorem}), we will always obtain one of these finitely many pregeometries in $\mathcal{G}_{n}$ under the assumption of instability of $x\nind_{C}y$ between types of rank $n$. This will complete the proof of our main theorem.

To show that Peretz's condition leads to finitely many pregeometries, we will extend a multi-experiment parameter identifiability argument based on \cite{Ovc22}, part of a larger research program initiated by Li, Meshkat, Ovchinnikov, Pillay, Pogudin and Scanlon within theoretical experimental biology (\cite{Li21}, \cite{Ovc21}, \cite{Li23}, \cite{Ovc22}, \cite{ovchinnikov2025identifiable}, \cite{meshkat2025algorithm}). The original multi-experiment argument involves indiscernible sequences whose terms are construed as undifferentiated entities. However, we will refine this argument to obtain additional information about indiscernible sequences whose terms are understood as sets of tuples, specifically tuples realizing some minimal partial type. The original argument, which first appears explicitly as Lemma 2.2 in \cite{LongKoponenConjecturePreprint}, is not explicitly published within the literature on multi-experiment parameter identifiability.\footnote{However, Proposition 7.27 from \cite{Ovc22} is closely related, also bounding the length of the initial segment $a_{0} \ldots a_{r-1}$ of a Morley sequence $\{a_{i}\}_{i < \omega}$ in a type $p(x) \in S(A)$ such that $\mathrm{Cb}(p)$ is contained within $\mathrm{acl}^{\mathrm{eq}}(a_{0}, \ldots, a_{n-1})$. This bound is given by the first $r$ such that $a_{r} \ind_{a_{0} \ldots a_{r-1}} A$, rather than by the rank of $\mathrm{Cb}(p)$. This appears relatively straightforward, because $\mathrm{tp}(a_{r}/Aa_{0}\ldots a_{r-1})$ is then a common nonforking extension of $p(x)|_{a_{0} \ldots a_{r-1}}$ and $p(x)$.   But the proposition, proven within the setting of differential fields, requires more work starting from the differential-equations background assumed in \cite{Ovc22}.} The most direct documented source for this argument is Corollary 3.2 of \cite{Ovc22}, which bounds the number of experiments required for parameter identifiability in terms of the number of parameters; the argument that we will refine will yield an abstract version of this corollary. One more informal source for the abstract argument itself is a set of discussions by Pillay and Scanlon communicated to the author in a personal communication of James Freitag, including a conference presentation of Pillay on his joint work with Ovchinnokov, Pogudin and Scanlon in \cite{Ovc22}.\footnote{This has led to the employment, including by the author, of the ``Pillay-Scanlon argument" as an informal nickname for the argument, though we believe the larger context of \cite{Li21}, \cite{Ovc21}, \cite{Li23}, \cite{Ovc22}, \cite{ovchinnikov2025identifiable}, \cite{meshkat2025algorithm} better explains the argument's full origins.}

The basis for the original multi-experiment argument is the following well-known fact (found in, say, Lemma 3.19 of Chapter 1 of \cite{P96}): for a type $p(x) \in S(A)$ and a Morley sequence $\{a_{i}\}_{i < \omega}$ of realizations of $p(x)$ over $A$, the canonical base $\mathrm{Cb}(p)$ of $p$ is contained within the algebraic closure $\mathrm{acl}^{\mathrm{eq}}(a_{0}, \ldots, a_{n-1})$ of some initial segment of the Morley sequence. (Here, we are using supersimplicity to get the canonical base contained in some initial segment $\mathrm{acl}^{\mathrm{eq}}(a_{0}, \ldots, a_{n-1})$, rather than just the full sequence $\mathrm{bdd}(\{a_{i}\}_{i < \omega})$ as in a general simple theory.) Speaking extremely roughly, the relevance to biologists of this fact comes from thinking of the canonical base $\mathrm{Cb}(p)$ as the parameter set of an experiment, and the Morley sequence $\{a_{i}\}_{i < \omega}$ as consisting of the results of a sequence of independent trials. That $\mathrm{Cb}(p) \subset \mathrm{acl}^{\mathrm{eq}}(a_{0}, \ldots, a_{n-1})$ then says that one can identify the parameters of an experiment from $n$ independent trials.

In practice, however, one may hope for some bound on the number of independent trials required for recovering the parameters of a given experiment. The arguments for Corollary 3.2 of \cite{Ovc22} supply such a bound. Specifically, we can view $\mathrm{SU}(\mathrm{Cb}(p(x)))$ as the number of independent parameters of an experiment. Then the arguments in \cite{Ovc22} show that the number of independent trials required to recover the parameters is bounded by the number of independent parameters. More formally, this multi-experiment argument shows:

\begin{fact}\label{Pillay-Scanlon argument}

Let $\{a_{i}\}_{i <\omega}$ be a Morley sequence over $A$ with $p(x) := \mathrm{tp}(a_{i}/A)$, and let $n = \mathrm{SU}(\mathrm{Cb}(p(x)))$. Then $\mathrm{Cb}(p) \subset \mathrm{acl}^{\mathrm{eq}}(a_{0}, \ldots a_{n-1})$, so $\{a_{i}\}_{n \leq i <\omega}$ is a Morley sequence over $a_{0}, \ldots a_{n-1}$.
    
\end{fact}

The main point of this argument is that, if $\mathrm{Cb}(p) \not\subset \mathrm{acl}^{\mathrm{eq}}(a_{0}, \ldots a_{k-1})$, $\mathrm{Cb}(p) \nind_{a_{0}, \ldots a_{k-1}} a_{k}$. So if this is the case for $0 \leq k \leq n$, $\mathrm{SU}(\mathrm{Cb}(p(x))) > n$, a contradiction.

That $\{a_{i}\}_{n \leq i <\omega}$ is a Morley sequence over $a_{0}, \ldots a_{n-1}$ tells us that, if we bound the rank of the canonical base of a type in which $\{a_{i}\}_{ i <\omega}$ is a Morley sequence (equivalently, of the limit type of $\{a_{i}\}_{ i <\omega}$), forking in $\{a_{i}\}_{ i <\omega}$ will be controlled by the forking in a bounded initial segment. However, in dealing with pregeometries, we will be considering sequences $\{a_{i}\}_{i <\omega}$ where the $a_{i}$ themselves consist of $n$ independent realizations $a^{0}_{i} \ldots a^{n-1}_{i}$ of an $\mathrm{SU}$-rank-$1$ partial type. We would like to get control over forking between the $a^{j}_{i}$, or in other words the pregeometry arising from these realizations of a minimal partial type, by bounding the rank of the canonical base of the limit type of $\{a_{i}\}_{i <\omega}$. For this, we will need a refinement of the multi-experiment argument. 

Before proving our refinement, we will first show that, given instability of $\nind_{C}$ between types of rank $n$, there is more than just a sequence $\{a_{i}, b_{i}\}_{i < \omega}$ exhibiting instability of $\nind_{C}$ between types of rank $n$ such that the canonical base of the limit type of $\{b_{i}\}_{i < \omega}$ is of bounded $\mathrm{SU}$-rank (as follows from the previous section's main lemma, Lemma \ref{existential Peretz theorem}). There is even a sequence $\{a_{i}, b_{i}\}_{i < \omega}$ exhibiting instability $\nind_{C}$ between types of rank $n$ such that the canonical base of the limit type of the sequence with the full terms $\{a_{i}, b_{i}\}_{i < \omega}$ is of bounded $\mathrm{SU}$-rank. (We will simultaneously apply the equidominance reduction from Fact \ref{reduction to sets of realizations of minimal types} to ensure that the terms are sets of realizations of a minimal partial type.) It is to this sequence that we will apply the refined multi-experiment argument.

\begin{proposition}\label{small canonical base}

Suppose that it is not the case that the forking relation is stable over a base between types of rank $n$. Then for some finite set $C$ there \textit{exists} a $C$-indiscernible sequence $\{a_i, b_i\}_{i < \omega}$ exhibiting the instability of $x \nind_{C} y$ (so $a_{i} \nind_{C} b_{j}$ exactly when $i > j$) such that the canonical base of the limit type of $\{a_i, b_i\}_{i < \omega}$ over $C$ has $\mathrm{SU}$-rank at most $n^{2}$ over $C$, and such that $a_{i}$, $b_{j}$ are each sets consisting of at most $n$ realizations of a minimal partial type over $C$.

\end{proposition}

\begin{proof}
    By Fact \ref{reduction to sets of realizations of minimal types}, there is some $C'$, and a $C'$-indiscernible sequence $\{a'_i, b'_i\}_{i< \omega}$, with $a'_{i} \nind_{C'} b'_{j}$ exactly when $i > j$ and $a'_{i} \models p(x)$, $b'_{i} \models q(y)$, where $p(x)$ and $q(y)$ are complete types realized by sets consisting of at most $n$ independent realizations of some minimal partial type over $C'$. So applying the remarks after Lemma \ref{existential Peretz theorem}, there is a base $C \supset C'$, and a $C$-indiscernible sequence $\{a_i, b_i\}_{i < \omega}$ exhibiting the instability of $x \nind_{C} y$ between types of rank $n$, such that the limit type of $\{b_i\}_{i < \omega}$ over $C$ does not fork over $Cb_{0} \ldots b_{n-2}$, and such that $a_{i} \models p(x)$, $b_{i} \models q(y)$. Thus $a_{i}$, $b_{j}$ are each sets consisting of at most $n$ independent realizations of a minimal partial type over $C$ (after dropping realizations of algebraic types over $C$). Without loss of generality, assume $C$ is empty.

    Let $B$ be the canonical base of the limit type of $\{ b_i\}_{i < \omega}$. Then because the limit type does not fork over $b_{0} , \ldots, b_{n-1}$, $B \subseteq \mathrm{acl}^{\mathrm{eq}}(b_{0} , \ldots, b_{n-1})$. So $\mathrm{SU}(B) \leq n^{2}$ by the bound arising from the Lascar equation, because $\mathrm{SU}(b_{i}) \leq n$. The following arguments involving canonical bases will be standard, but we give them anyway, using only that fact that $\mathrm{Cb}(p)$ for $p \in S(B)$ is contained in $\mathrm{acl}^{\mathrm{eq}}(B)$, and the fact that canonical bases are preserved under nonforking extensions. Extend $\{ b_i\}_{i < \omega}$ to an indiscernible sequence $\{ b_i\}_{i \in \mathbb{Q}}$; its \textit{limit type} will be the type over  $\{ b_i\}_{i \in \mathbb{Q}}$ consisting of formulas satisfied by $b_{j}$ for arbitrarily large $j$. Notice the limit type of $\{ b_i\}_{i \in \mathbb{Q}}$ does not fork over $\{ b_i\}_{i \in \mathbb{Z}_{\geq 0}}$ (being finitely satisfiable in $\{ b_i\}_{i \in \mathbb{Z}_{\geq 0}}$). So the canonical base of the limit type of $\{ b_i\}_{i \in \mathbb{Q}}$ is also $B$. Moreover, the limit type of $\{ b_i\}_{i \in \mathbb{Q}}$ does not fork over $\{b_{i}\}_{k \leq i \in \mathbb{Q}}$ for each $k < \omega$,  being finitely satisfiable there. So $B \subseteq \bigcap_{k \leq \omega} \mathrm{acl}^{\mathrm{eq}}(\{b_{i}\}_{k \leq i \in \mathbb{Q} })$, and $\{ b_i\}_{i \in \mathbb{Q}}$ is indiscernible over $B$. In fact, $\{ b_i\}_{i \in \mathbb{Q}}$ is a Morley sequence over $B$: for $b_{\infty}$ realizing the limit type of $\{ b_i\}_{i \in \mathbb{Q}}$ over $B$, we must show $b_{\infty} \ind_{B} \{ b_i\}_{i \in \mathbb{Q}}$. Because $B \subseteq \mathrm{acl}^{eq}(\{ b_i\}_{i \in \mathbb{Q}})$, the canonical base of $\mathrm{tp}(b_{\infty}/ B \{ b_i\}_{i \in \mathbb{Q}})$ will be the same as the canonical base of $\mathrm{tp}(b_{\infty}/ \{ b_i\}_{i \in \mathbb{Q}})$, the limit type of $\{ b_i\}_{i \in \mathbb{Q}}$. So the canonical base of $\mathrm{tp}(b_{\infty}/ B \{ b_i\}_{i \in \mathbb{Q}})$ will be $B$, and $b_{\infty} \ind_{B} \{ b_i\}_{i \in \mathbb{Q}}$.

    We may find, similarly to the beginning of the proof of Theorem 4.1 of \cite{PW13} or the beginning of the proof of Lemma \ref{existential Peretz theorem} above, some $a$ with $\{b_{i}\}_{i > 0}$ indiscernible over $a\{b_{i}\}_{i \leq 0}$, $\{b_{i}\}_{i \leq 0}$ indiscernible over $a\{b_{i}\}_{i > 0}$ (where $i$ ranges over $\mathbb{Q}$), $a\nind b_0$, and $a \ind b_1$. Because $B \subseteq \bigcap_{k \leq \omega} \mathrm{acl}^{\mathrm{eq}}(\{b_{i}\}_{k \leq i})$, it is also the case that  $\{b_{i}\}_{i > 0}$ is indiscernible over $Ba\{b_{i}\}_{i \leq 0}$, and $\{b_{i}\}_{i \leq 0}$ is indiscernible over $B a\{b_{i}\}_{i > 0}$. Similarly to the beginning of the proof of Theorem 4.1 of \cite{PW13}, we show that $a \ind_{B} \{b_{i}\}_{i \in \mathbb{Q}}$. By transitivity and finite character, it suffices to show that for all finite $S \subset \mathbb{Q}$ and $i \in \mathbb{Q} \backslash S$, $a \ind_{B\{b_{j}\}_{j \in S}} b_{i}$. By symmetry and base monotonicity, it suffices to show that $b_{i} \ind_{B} a \{b_{j}\}_{j \in S} $. By the instances of indiscernibility, $\mathrm{tp}(b_{i}/ a B \{b_{j}\}_{j < i}\{b_{j}\}_{j \in S})$ is finitely satisfiable in $\{b_{j}\}_{j < i}$, where $j$ ranges over $\mathbb{Q}$. So $b_{i} \ind_{B\{b_{j}\}_{j < i}} a \{b_{j}\}_{j \in S} $. Then by transitivity, to show $b_{i} \ind_{B} a \{b_{j}\}_{j < i} \{b_{j}\}_{j \in S} $ and thus $b_{i} \ind_{B} a \{b_{j}\}_{j \in S} $, it suffices to show $b_{i} \ind_{B} \{b_{j}\}_{j < i}  $. But this is just the fact that  $\{ b_i\}_{i \in \mathbb{Q}}$ is a Morley sequence over $B$. 

    Now for $i \in \mathbb{Q}$, let us choose $a^{*}_{i}$ such that $a^{*}_{i} \{b_{j}\}_{j \leq i} \{b_{j}\}_{j > i} \equiv_{B} a \{b_{j}\}_{j \leq 0} \{b_{j}\}_{j > 0} $. (More explicitly, $a^{*}_{i} b_{j_{1}+i} \ldots b_{j_{n}+i}  \equiv_{B} a b_{j_{1}} \ldots b_{j_{n}} $ for $j_{1} < \ldots < j_{n} \in \mathbb{Q} $.) Moreover, choose the $a^{*}_{i}$ so that they form an independent set over $B \{b_{i}\}_{i \in \mathbb{Q}}$. Note that by symmetry, this last condition is equivalent to $a^{*}_{i_{1}} \ldots a^{*}_{i_{n}} \ind_{B \{b_{i}\}_{i \in \mathbb{Q}}} a^{*}_{i} $ for $i_{1}< \ldots < i_{n} < i $. That is the same thing, relative to $a^{*}_{i} \{b_{j}\}_{j \leq i} \{b_{j}\}_{j > i} \equiv_{B} a \{b_{j}\}_{j \leq 0} \{b_{j}\}_{j > 0} $, as saying $a^{*}_{i_{1}} \ldots a^{*}_{i_{n}}$ do not satisfy a formula of the form $\varphi(x, b_{j_{1}} \ldots b_{j_{m}} a_{*}^{i})$, for $\varphi(x, y)$ a formula with parameters in $B$, where $\varphi(x, b_{j_{1}-i} \ldots b_{j_{m}-i} a)$ forks over $B \{b_{i}\}_{i \in \mathbb{Q}}$. So independence of the $a^{*}_{i}$ over $B \{b_{i}\}_{i \in \mathbb{Q}}$ is implied by the EM-type of $\{a^{*}_{i} b_{i}\}_{i < \omega}$ over $B$ (see Subclaim 7.12 of \cite{LongKoponenConjecturePreprint} for a similar argument). Thus we can even find $a^{*}_{i}$ with $a^{*}_{i} \{b_{j}\}_{j \leq i} \{b_{j}\}_{j > i} \equiv_{B} a \{b_{j}\}_{j \leq 0} \{b_{j}\}_{j > 0} $, independent over $B \{b_{i}\}_{i \in \mathbb{Q}}$, such that $\{a^{*}_{i} b_{i}\}_{i < \omega}$ forms a $B$-indiscernible sequence.
    
    By $a^{*}_{i} \{b_{j}\}_{j \leq i} \{b_{j}\}_{j > i} \equiv_{B} a \{b_{j}\}_{j \leq 0} \{b_{j}\}_{j > 0} $, $a\nind b_0$, and $a \ind b_1$,  $\{a^{*}_{i} b_{i}\}_{i < \omega}$ will exhibit instability of $\nind$ between types realized by at most $n$ realizations of a minimal partial type. That is, $a^{*}_{i} \nind b_{j}$ exactly when $i > j$, and $a^{*}_{i}$, $b_{i}$ are sets consisting of at most $n$ realizations of a minimal partial type. Moreover, by $a^{*}_{i} \{b_{j}\}_{j \leq i} \{b_{j}\}_{j > i} \equiv_{B} a \{b_{j}\}_{j \leq 0} \{b_{j}\}_{j > 0} $ and $a \ind_{B} \{b_{i}\}_{i \in \mathbb{Q}}$, $a^{*}_{i} \ind_{B} \{b_{i}\}_{i \in \mathbb{Q}}$ for $i \in \mathbb{Q}$. By this and independence of the $a^{*}_{i}$ over $B \{b_{i}\}_{i \in \mathbb{Q}}$ (as well as symmetry, transitivity, base monotonicity), $\{a^{*}_{i} b_{i}\}_{i < \omega}$ forms a Morley sequence over $B$. So its limit type over $B$ does not fork over $\{a^{*}_{i} b_{i}\}_{i < \omega}$, by finite satisfiability, nor over $B$. Thus the canonical base of the limit type of $\{a^{*}_{i} b_{i}\}_{i < \omega}$ is $B$. Because $\mathrm{SU}(B) \leq n^{2}$, the canonical base of the limit type of $\{a^{*}_{i} b_{i}\}_{i < \omega}$  has $\mathrm{SU}$-rank at most $n^{2}$, as desired.

\end{proof}

We see now that if the forking relation is unstable over a base between types of rank $n$, there is some base $C$ and $C$-indiscernible sequence $\{a_{i},b_{i}\}_{i < \omega}$ exhibiting instability of $\nind_{C}$, such that the terms $a_{i}b_{i}$ are sets consisting of a bounded number (at most $2n$) realizations of a minimal partial type over $C$, and the canonical base of the limit type of $\{a_{i},b_{i}\}_{i < \omega}$ over $C$ has bounded $\mathrm{SU}$-rank (at most $n^{2}$) over $C$. So we have obtained, as an instance of instability for $\nind_{C}$, the kind of indiscernible sequence that we said we wanted to examine in our refinement of the multi-experiment argument in Fact \ref{Pillay-Scanlon argument}: an indiscernible sequence whose terms themselves consist of a bounded number of realizations of a minimal partial type, and such that the canonical base of the limit type has bounded $\mathrm{SU}$-rank. We now prove our refinement, controlling the forking between realizations of a minimal type within these sequences.

Without loss of generality, we restrict our attention to indiscernible sequences over the empty set. Let $p(x) \in S(\emptyset)$ be a minimal partial type and let $\{a_{i}\}_{i < \omega}$ be an indiscernible sequence, where for $i < \omega$ the $a_{i} = \{a^{j}_{i}\}^{n}_{j = 1}$ are themselves sets consisting of at most $n$ realizations of $p(x)$. (Without loss of generality we use the notation $a_{i} = \{a^{j}_{i}\}^{n}_{j = 1}$ denoting exactly $n$ realizations.) Let $S := \omega \times \{1, \ldots n\}$; then we can identify the induced sub-pregeometry of $p(x)$ on all of the $a^{j}_{i}$ taken together (for all $i < \omega$, $1 \leq j \leq n$) with a pregeometry on $S$. Here we just use the map $(i, j) \mapsto a^{j}_{i}$ for $(i, j) \in S$. Then for two indiscernible sequences (possibly in two different theories) whose terms are sets consisting of at most $n$ realizations of a minimal partial type, we say their pregeometries are \textit{equivalent} if they yield the same pregeometry on $S$. We may now state the objective of the refined multi-experiment argument.

\begin{lemma}\label{refined Pillay-Scanlon argument}

Let $n, m < \omega$ be fixed values. Consider all of the theories $T$, and indiscernible sequences $\{a_{i}\}_{i< \omega}$, satisfying the following:

\begin{itemize}
    \item the terms $a_{i}$ are sets consisting of at most $n$ realizations $a^{j}_{i}$ of a minimal partial type, and

    \item the canonical base of the limit type of $\{a_{i}\}_{i< \omega}$ has $\mathrm{SU}$-rank at most $m$.
\end{itemize}

Then among these, there are only finitely many pregeometries on the $a^{i}_{j}$ up to equivalence. (Notice from the statement here that these pregeometries will depend only on $n$ and $m$, not the theory $T$ and indiscernible sequence $\{a_{i}\}_{i< \omega}$.)

\end{lemma}

\begin{proof}

Let $n, m < \omega$ be fixed and $\{a_{i}\}_{i < \omega}$ an indiscernible sequence as in the statement. Without loss of generality, for notational convenience we assume $a_{i}$ consists of exactly $n$ realizations $a^{1}_{i} \ldots a^{n}_{i}$ of a minimal partial type. To this indiscernible sequence, we associate a finite tree of height at most $m$ whose non-root nodes are labeled by nonempty subsets of $\{1, \ldots, n\}$. For $S \subseteq \{1, \ldots, n\}$, $i < \omega$, let $a^{S}_{i} := \{a^{j}_{i}\}_{j \in S} \subset a_{i}$. Let $A$ be the canonical base of the limit type of $\{a_{i}\}_{i < \omega}$; just as in the proof of the previous proposition, $\{a_{i}\}_{i < \omega}$ is a Morley sequence over $A$. Define the tree inductively as follows: from the root node, let the immediate successors consist of one node each labeled by each subset $S \subseteq \{1, \ldots, n\}$ for which $A \nind a^{S}_{0}$ (with an arbitrary left-to-right order on the immediate successors). Now assume the tree has been defined up to height $k < \omega$. We define the immediate successors of a height-$k$ node labeled by $S_{k-1}$, ending the path in the tree the labels of whose successive non-root nodes are $S_{0}, \ldots, S_{k-1}$. These immediate successors (again, ordered arbitrarily from left to right) will consist of one node each labeled by each subset $S \subseteq \{1, \ldots, n\}$ for which $A\nind_{a^{S_{0}}_{0}\ldots a_{k-1}^{S_{k-1}}} a^{S}_{k}$. This tree must have height at most $m$. To see this, otherwise from a path to a node of height $m+1$, we will obtain a chain of $m+1$ forking extensions starting with $\mathrm{tp}(A)$. And this will contradict that $\mathrm{SU}(A) \leq m$. (This last observation is essentially the content of the original multi-experiment argument.)

For each of the finitely many trees $\mathcal{T}$ of height at most $m$, such that the non-root nodes are each labeled by nonempty subsets of $\{1, \ldots , n\}$, and the immediate successors of a fixed node are labeled with distinct subsets, we associate to the tree $\mathcal{T}$ a function $f_{\mathcal{T}} : \mathcal{P}(\omega \times \{1, \ldots, n\}) \to \mathcal{P}(\omega \times \{1, \ldots, n\})$. This function will satisfy the assertion that for each $S \subset \omega \times \{1, \ldots, n\}$, $f_{\mathcal{T}}(S) \subset S$. In fact, $f_{\mathcal{T}}(S)$ will be the union of at most $m$ many sets of the form $(\{i\} \times \{1, \ldots n\})\cap S$ for $i < \omega$. To define $f(S)$, we define:

\begin{itemize}

\item A path through $\mathcal{T}$, starting with the root, whose $i$th node is always an immediate successor of the $(i-1)$st node; this last node of this path may or may not  be a leaf node of the tree. We will use $T_{1}, \ldots T_{k}$ to denote the labels of the $k \leq m$ non-root nodes of this path (in the order of these nodes along the path, so this list will determine the path.)

\item An increasing sequence $i_{1} < \ldots < i_{k} < \omega$.

\end{itemize}

To define the first node of this path and to define $i_{1}$, let $i_{1}$ be the smallest $i_{1} < \omega$, if it exists, for which $S^{i_{1}} := \{\ell \in \{1, \ldots n\}: (i_{1}, \ell) \in S \}$ is a label of some immediate successor of the root node. Then choose the immediate successor labeled by $S^{i_{1}} =: T_{1}$. If such $i_{1}$ does not exist, terminate the path at the root node. Let us now choose $T_{j+1}$ and $i_{j+1}$, given $T_{1}, \ldots T_{j}$ for $j < k$ determining the first $j$ non-root nodes of the path, and $i_{1} < \ldots < i_{j} < \omega$. To find this label, let $i_{j+1}$ be the smallest $i_{j+1} < \omega$ with $i_{j+1} > i_{j}$, if such $i_{j+1}$ exists, for which $S^{i_{j+1}} = \{\ell \in \{1, \ldots n\}: (i_{j+1}, \ell) \in S \}$ is a label of some immediate successor of the last node on the path with labels $T_{1}, \ldots T_{j}$. Similarly to the first step, choose the immediate successor labeled by $S^{i_{j+1}} =: T_{j+1}$, and if such $i_{j+1}$ does not exist, terminate the construction of the path. Then $f(S)$ will be defined as follows:

$$f(S) := \bigsqcup^{k}_{j = 1} (\{i_{j}\} \times S^{i_{j}} ) .$$

To summarize this construction informally, we found the first ``slice" of $S$ representing an immediate successor of the root node, added it to $f(S)$ and went to that immediate successor, and continued in this manner: at each step we found the next ``slice" of $S$, above the slices we already added, representing an immediate successor of our current node, and added it to $f(S)$ and went to that immediate successor. We ended upon reaching a leaf node, or a node with no immediate successors represented by any ``slice" above the ones we already added.

Now for $S \subset \omega \times \{1, \ldots n\}$, let $a^{S} := \{a^{j}_{i}: (i, j) \in S\}$. We will have proven our lemma if we can show that the assignment of $\mathrm{SU}$-ranks to finite sets from among all of the $a^{j}_{i}$, or equivalently the pregeometry on the $a^{j}_{i}$, is determined by the following:

\begin{itemize}

\item the function $f_{\mathcal{T}}$

\item the $\mathrm{SU}$-ranks of the sets of $a^{j}_{i}$ contained in a union of $m+1$ of the $a_{i}$

\end{itemize}

Note that there are finitely many possibilities for $f_{\mathcal{T}}$ because (as we argued using the main idea of the multi-experiment argument) the tree $\mathcal{T}$ must have height at most $m$, so there are finitely many possibilities for $\mathcal{T}$. There are finitely many possibilities for the  $\mathrm{SU}$-ranks of sets contained in $m+1$ terms, by indiscernibility. So it really is the case that we just need to show that all of the $\mathrm{SU}$-ranks of sets consisting of finitely many $a^{j}_{i}$ are determined by the above.

This will be the case if we can show that, for finite $S \subset \omega \times \{1, \ldots, n\}$, and $a^{S}_{i}$ the set of $a^{j}_{i}$ for which $(i, j) \in S$, the $a^{S}_{i}$ for which $S \cap (\{i\} \times \{1, \ldots n\}) \not \subset f(S)$ form an independent set over $a^{f(S)}$. (Note that $a^{f(S)}$ is the rest of $a^{S}$ outside these $a^{S}_{i}$, so the disjoint union of those $a^{S}_{i}$ such that $S \cap (\{i\} \times \{1, \ldots n\} ) \subset f(S)$.) Specifically, when $S \cap ( \{i\} \times \{1, \ldots n\} ) \not \subset f(S)$, $\mathrm{SU}(a^{S}_{i}a^{f(S)})$ will be the $\mathrm{SU}$-rank of a set contained in a union of $m+1$ of the $a_{\ell}$. So if these $a^{S}_{i}$ form an independent set over $a^{f(S)}$, $\mathrm{SU}(a^{S})$ will be determined from the $\mathrm{SU}(a^{S}_{i}a^{f(S)})$ by the Lascar equation. We will then be done, because $a^{S}$ is an arbitrary finite subset from among the $a^{j}_{i}$.

To show that the $a^{S}_{i}$ for which $S \cap \{i\} \times \{1, \ldots n\} \not \subset f(S)$ are independent over $a^{f(S)}$, it suffices to show that for each such $a^{S}_{i}$, $a^{S}_{i} \ind_{a^{f(S)}} \{a_{j}\}_{j < \omega, j\neq i }$. But by construction, and $A$-indiscernibility of $\{a_{i}\}_{i < \omega}$, $A\ind_{a^{(f(S) \cap (\{0, \ldots, i-1\} \times \{1, \ldots n\}))}}  a^{S}_{i} $. So by symmetry, $a^{S}_{i} \ind_{a^{(f(S) \cap (\{0, \ldots, i-1\} \times \{1, \ldots n\}))}} A $. Because $\{a_{j}\}_{j < \omega}$ is a Morley sequence over $A$ (and symmetry, transitivity, base monotonicity), $ a^{S}_{i} \ind_{A} \{a_{j}\}_{j < \omega, j\neq i }$. Thus, by base monotonicity, $ a^{S}_{i} \ind_{Aa^{(f(S) \cap (\{0, \ldots, i-1\} \times \{1, \ldots n\}))}} \{a_{j}\}_{j < \omega, j\neq i }$. So by transitivity,  $ a^{S}_{i} \ind_{a^{(f(S) \cap (\{0, \ldots, i-1\} \times \{1, \ldots n\}))}} \{a_{j}\}_{j < \omega, j\neq i }$. And by base monotonicity, $a^{S}_{i} \ind_{a^{f(S)}} \{a_{j}\}_{j < \omega, j\neq i }$, as desired.

\end{proof}

We are now ready to conclude. To prove Theorem \ref{main theorem, restated}, let us start by taking the potentially infinite set $\mathcal{G}_{n}$ as constructed before, within the proof of Fact \ref{standard pregeometric interpretation, restated}. Recall that this was the set consisting of pregeometries $(S, \mathrm{cl})$, where $S$ can be written as a union $\bigcup_{i < \kappa} (a_{i} \cup b_{i})$ for sets $a_{i}, b_{i}$ of size at most $n$, such that $\mathrm{dim}(a_{i} \cup a_{j}) < \mathrm{dim}(a_{i}) + \mathrm{dim}(a_{j})$ exactly when $i > j$.  (Remember $\kappa$ was the cardinal large enough to apply the Erdős–Rado theorem over a finite set.) And our assertion from Fact \ref{standard pregeometric interpretation, restated} about these pregeometries was that a failure of $x\nind_{C}y$ to be stable for some base $C$ between types of rank $n$ was equivalent to the pregeometry on some minimal partial type over a finite set embedding one of the $\mathcal{G}_{n}$. So, it will be enough to find finitely many pregeometries \textit{within this $\mathcal{G}_{n}$} such that, assuming that the forking relation is unstable over a base between types of rank $n$, there is some minimal partial type over a finite set whose pregeometry embeds one of these finitely many pregeometries.

Suppose that the forking relation is unstable over a base between types of rank $n$. Then by Proposition \ref{small canonical base}, there exists a finite set $C$ and a $C$-indiscernible sequence $\{a_i, b_i\}_{i < \omega}$ such that $a_{i} \nind_{C} b_{j}$ exactly when $i > j$, such that the canonical base of the limit type of $\{a_i, b_i\}_{i < \omega}$ over $C$ has $\mathrm{SU}$-rank at most $n^{2}$ over $C$, and such that $a_{i}$, $b_{j}$ are each sets consisting of at most $n$ realizations of a minimal partial type $p(x)$ over $C$. By Lemma \ref{refined Pillay-Scanlon argument}, there are finitely many possibilities depending only on $n$, up to equivalence, for the pregeometry on all of the realizations $a^{j}_{i}$, $b^{j}_{i}$ of $p(x)$ within the $a_{i}$, $b_{i}$. Let $\{a_i, b_i\}_{i < \kappa}$ be a $C$-indiscernible sequence extending $\{a_i, b_i\}_{i < \omega}$. Then the pregeometry of the realizations $a^{j}_{i}$, $b^{j}_{i}$ within the terms $\{a_i, b_i\}_{i < \kappa}$ gives an embedding into the pregeometry on $p(x)$ of one of the members of $\mathcal{G}_{n}$, as in the proof of Fact \ref{standard pregeometric interpretation, restated}. But because $\{a_i, b_i\}_{i < \kappa}$ is a $C$-indiscernible sequence extending the $C$-indiscernible sequence  $\{a_i, b_i\}_{i < \omega}$, the pregeometry on the $a^{j}_{i}$, $b^{j}_{i}$ for $i < \omega$ determines the pregeometry on the $a^{j}_{i}$, $b^{j}_{i}$ for $i < \kappa$. So the member of $\mathcal{G}_{n}$ embedding into the pregeometry on $p(x)$ is determined by the pregeometry on $a^{j}_{i}$, $b^{j}_{i}$ for $i < \omega$ up to equivalence. Thus there are finitely many possibilities for this embedded pregeometry in $\mathcal{G}_{n}$, as desired.

\section{Canonical bases}\label{canonical base section}

One way of representing an instance of the instability of the forking relation $\nind_{C}$ over a base $C$ is by a parameter $a$ and $C$-indiscernible sequence $\{b_{i}\}_{i \in (-\infty, -1 ] \sqcup [1, \infty) \subseteq \mathbb{Q}}$, where $\{b_{i}\}_{i \geq 1}$ is $Ca\{b_{i}\}_{i \leq -1}$-indiscernible, $\{b_{i}\}_{i \leq -1}$ is $Ca\{b_{i}\}_{i \geq 1}$-indiscernible, $a \nind_{C} b_{-1}$, and $a \ind_{C} b_{1}$. (This corresponds to the beginning of the proof of Theorem 4.1 of \cite{PW13}, where $b_{-1}$, $b_{1}$ are essentially the realizations of the limit types discussed there.) As in the standard arguments within the proof of Proposition \ref{small canonical base}, when $\mathrm{Cb}$ is the canonical base of the limit type (in the positive direction) of $\{b_{i}\}_{i \in (-\infty, -1 ] \sqcup [1, \infty)}$ over $C$, $\{b_{i}\}_{i \in (-\infty, -1 ] \sqcup [1, \infty)}$ will be a Morley sequence over $C\mathrm{Cb}$. So $\mathrm{Cb}$ will also be the canonical base of the limit type in the negative direction of $\{b_{i}\}_{i \in (-\infty, -1 ] \sqcup [1, \infty)}$ over $C$, the canonical base of $\mathrm{tp}(b_{-1}/C\{b_{i}\}_{i \in (-\infty, -1 ) \sqcup [1, \infty)})$, and the canonical base of $\mathrm{tp}(b_{1}/C\{b_{i}\}_{i \in (-\infty, -1 ] \sqcup (1, \infty)})$. The argument for these claims will also be standard, and will be similar in all of the cases. For example, for $\mathrm{tp}(b_{-1}/C\{b_{i}\}_{i \in (-\infty, -1 ) \sqcup [1, \infty)})$, $\mathrm{tp}(b_{-1}/C\mathrm{Cb}\{b_{i}\}_{i \in (-\infty, -1 ) \sqcup [1, \infty)})$ is a common nonforking extension of $\mathrm{tp}(b_{-1}/C\mathrm{Cb})$, and also of $\mathrm{tp}(b_{-1}/C\{b_{i}\}_{i \in (-\infty, -1 ) \sqcup [1, \infty)})$ by finite satisfiability, so these last two have the same canonical base $\mathrm{Cb}$.

In our existential generalization of Peretz's theorem to higher ranks, Lemma \ref{existential Peretz theorem}, we show that if the forking relation is unstable over a base between types of rank $n$, there is a base $C$ and $a$, $\{b_{i}\}_{i \in (-\infty, -1 ] \sqcup [1, \infty)}$ exhibiting instability of $\nind_{C}$ between types of rank $n$ such that, say, $\mathrm{Cb} \subset \mathrm{acl}(Cb_{1}, \ldots, b_{k}, \ldots b_{n-1})$ (for integers $k$ with $1 \leq k \leq n-1$). A shortcoming of Lemma \ref{existential Peretz theorem} as stated is that its conclusion is dependent on the rank $n$. We will begin this section by restating our generalization of Peretz's results in a way that is independent of the rank.

\begin{proposition}\label{existential Peretz theorem, rank-independent statement}

In a supersimple theory of finite $\mathrm{SU}$-rank, suppose that the forking relation is unstable over a base. Then there is some base $C$, and  $a$, $\{b_{i}\}_{i \in (-\infty, -1 ] \sqcup [1, \infty)}$ exhibiting instability of $\nind_{C}$, such that the following is satisfied. Let $\mathrm{Cb}_{-}$ be the smallest $\mathrm{acl}_{C}^{\mathrm{eq}}$-closed subset of $\mathrm{acl}^{\mathrm{eq}}(\{b_{i}\}_{i \in (-\infty, -1 ) \sqcup [1, \infty)})$ such that  $a \ind_{\mathrm{Cb}_{-}} b_{-1}$; let $\mathrm{Cb}_{+}$ be the smallest $\mathrm{acl}_{C}^{\mathrm{eq}}$-closed subset of $\mathrm{acl}^{\mathrm{eq}}(\{b_{i}\}_{i \in (-\infty, -1 ] \sqcup (1, \infty)})$ such that  $a \ind_{\mathrm{Cb}_{+}} b_{1}$ and $a \nind_{C} \mathrm{Cb}_{+}$. Then $\mathrm{Cb}_{+}$, $\mathrm{Cb}_{-}$ are well-defined and equal to $\mathrm{acl}_{C}^{\mathrm{eq}}(\mathrm{Cb})$.
    
\end{proposition}

\begin{proof}

As in the proof of Lemma \ref{existential Peretz theorem}, let 
$C$,  $a$, $\{b_{i}\}_{i \in (-\infty, -1 ] \sqcup [1, \infty)}$ exhibit instability of $\nind_{C}$ such that:

\begin{itemize}
    \item $\mathrm{SU}(a/C)$ is minimized, and

    \item  subject to that minimization, $\mathrm{SU}(\mathrm{Cb}(b_{i}/\mathrm{Cb}))=\mathrm{SU}(\mathrm{Cb}(b_{i}/C\mathrm{Cb}))$ (the $\mathrm{SU}$-rank of the limit type over $C$) is maximized.
\end{itemize}

We show this is as desired. 

To start, $\mathrm{acl}_{C}^{\mathrm{eq}}(\mathrm{Cb})$ will itself satisfy the conclusion for $\mathrm{Cb}_{-}$, and $\mathrm{Cb}_{+}$ (before we show it is the smallest set such that this conclusion is satisfied). For $\mathrm{Cb}_{-1}$, this means that $a \ind_{\mathrm{acl}_{C}^{\mathrm{eq}}(\mathrm{Cb})} b_{-1}$, or equivalently $a \ind_{C\mathrm{Cb}} b_{-1}$. For $\mathrm{Cb}_{+}$, this means that $a \ind_{\mathrm{acl}_{C}^{\mathrm{eq}}(\mathrm{Cb})} b_{1}$, or equivalently $a \ind_{C\mathrm{Cb}} b_{1}$, and that $a \nind_{C}\mathrm{acl}_{C}^{\mathrm{eq}}(\mathrm{Cb})$, or equivalently $a \nind_{C}\mathrm{Cb}$. Let us start with the conditions for $\mathrm{Cb}_{-}$. We will argue as at the beginning of the proof of Theorem 4.1 of \cite{PW13}. By symmetry, our desired condition $a \ind_{C\mathrm{Cb}} b_{-1}$ is equivalent to $b_{-1} \ind_{C\mathrm{Cb}} a $. This will be true if $b_{-1} \ind_{C\mathrm{Cb}} a \{b_{i}\}_{i \in (-\infty, -1 ) \sqcup [1, \infty)}$. But $b_{-1} \ind_{C \{b_{i}\}_{i \in (-\infty, -1 ) \sqcup [1, \infty)} } a$ by finite satisfiability, and $b_{-1} \ind_{C\mathrm{Cb}} \{b_{i}\}_{i \in (-\infty, -1 ) \sqcup [1, \infty)}$ by the observations about $\mathrm{Cb}$ before this proposition's statement. So $b_{-1} \ind_{C\mathrm{Cb}} a \{b_{i}\}_{i \in (-\infty, -1 ) \sqcup [1, \infty)}$ by transitivity, as desired. Let us now prove that $\mathrm{Cb}$ satisfies the conditions for $\mathrm{Cb}_{-}$. The condition $a \ind_{C\mathrm{Cb}} b_{1}$ is as before. And the condition $a \nind_{C}\mathrm{Cb}$ follows by transitivity from $a \ind_{C\mathrm{Cb}} b_{-1}$ and $a \nind_{C} b_{-1}$.

In fact, the rest of the proof for $\mathrm{Cb}_{+}$ will be readily apparent from the rest of the proof for $\mathrm{Cb}_{-}$. So we will only give the rest of the proof for $\mathrm{Cb}_{-}$.

Let $\mathrm{Cb}' \subset \mathrm{acl}_{C}^{\mathrm{eq}}(\{b_{i}\}_{i \in (-\infty, -1 ) \sqcup [1, \infty)})$ be a $\mathrm{acl}_{C}^{\mathrm{eq}}$-closed set satisfying $a \ind_{C\mathrm{Cb}'} b_{-1}$. We show that $\mathrm{Cb}' \supset \mathrm{Cb}$, which will be what we need. We may extend $\{b_{i}\}_{i \in (-\infty, -1 ) \sqcup [1, \infty)}$ to a $C$-indiscernible sequence $\{b_{i}\}_{i \in \mathbb{Q}_{\neq 0}}$ such that $\{b_{i}\}_{i > 0}$ is $Ca\{b_{i}\}_{i < 0}$-indiscernible, $\{b_{i}\}_{i < 0}$ is $Ca\{b_{i}\}_{i >0}$-indiscernible, $a \nind_{C} b_{i}$ for $i < 0$, and $a \ind_{C} b_{i}$ for $i > 0$. Then $\{b_{i}\}_{i \in [-1, 0) \sqcup (0, 1)}$ will be indiscernible over $\mathrm{Cb}'$, because $\mathrm{Cb}' \subset \mathrm{acl}_{C}^{\mathrm{eq}}(\{b_{i}\}_{i \in (-\infty, -1 ) \sqcup [1, \infty)})$.

If $\mathrm{Cb}' \not \supset \mathrm{Cb}$, then $\mathrm{SU}(b_{-1}/\mathrm{Cb}') > \mathrm{SU}(b_{-1}/C\mathrm{Cb}) $. To see this, suppose $\mathrm{SU}(b_{-1}/\mathrm{Cb}') = \mathrm{SU}(b_{-1}/C\mathrm{Cb}) $. First, $\mathrm{tp}(b_{-1}/C\mathrm{Cb}\{b_{i}\}_{i \in (-\infty, -1 ) \sqcup [1, \infty)})$, so $\mathrm{tp}(b_{-1}/\mathrm{acl}_{C}^{\mathrm{eq}}(\{b_{i}\}_{i \in (-\infty, -1 ) \sqcup [1, \infty)}))$, is a nonforking extension of $\mathrm{tp}(b_{-1}/C\mathrm{Cb})$. So by $\mathrm{SU}(b_{-1}/\mathrm{Cb}') = \mathrm{SU}(b_{-1}/C\mathrm{Cb}) $, $\mathrm{tp}(b_{-1}/C\mathrm{Cb})$ and $\mathrm{tp}(b_{-1}/\mathrm{Cb}')$ will have a common nonforking extension. Thus the canonical base of $\mathrm{tp}(b_{-1}/\mathrm{Cb}')$ will be the same as the canonical base of $\mathrm{tp}(b_{-1}/C\mathrm{Cb})$, so equal to $\mathrm{Cb}$. Thus  $\mathrm{Cb}'   \supset \mathrm{Cb}$, a contradiction.

Moreover, $\mathrm{SU}(a /\mathrm{Cb}') < \mathrm{SU}(a/C)$. In other words, $a \nind_{C} \mathrm{Cb}' $. Otherwise, by $a \nind_{C} b_{-1}$, $a \ind_{\mathrm{Cb}'} b_{-1}$ and transitivity, we will have a contradiction.

By assumption, $a \ind_{\mathrm{Cb}'} b_{-1}$. Let $j \in (0, 1)$ be arbitrary. We are in one of two cases: either $a \nind_{\mathrm{Cb}'} b_{j}$ or $a \ind_{\mathrm{Cb}'} b_{j}$.

In the case $a \nind_{\mathrm{Cb}'} b_{j}$, we can conclude as in case 1 of the proof of Lemma \ref{existential Peretz theorem}, by replacing the base $C$ with $\mathrm{Cb}'$ and flipping the $\mathrm{Cb}'$-indiscernible sequence $\{b_{i}\}_{i \in [-1, 0) \sqcup (0, 1)}$. Here is where we use that $\mathrm{SU}(a /\mathrm{Cb}') < \mathrm{SU}(a/C)$, getting a contradiction to minimality of $\mathrm{SU}(a/C)$.

In the case $a \ind_{\mathrm{Cb}'} b_{j}$, we can conclude as in Case 2 of the proof of Lemma \ref{existential Peretz theorem}, by repeated applications of the independence theorem. Here, we use $\mathrm{SU}(b_{-1}/\mathrm{Cb}') > \mathrm{SU}(b_{-1}/C\mathrm{Cb}) $, getting a contradiction to maximality of  $\mathrm{SU}(b_{-1}/C\mathrm{Cb})$.

\end{proof}

We have completed our rank-indpenendent statement of our existential, higher-rank generalization of Peretz's theorem.

As we noted in Section \ref{existential Peretz theorem}, we have not yet ruled out the possibility that we can recover a \textit{universal} generalization of Peretz's theorem to higher ranks, by taking canonical bases or otherwise shrinking terms.

\begin{question}\label{term shrinking question}

In a finite-rank supersimple theory, for $\{a_{i}, b_{i}\}_{i < \omega}$ a $C$-indiscernible sequence exhibiting instability of $\nind_{C}$ between types of rank at most $n$, let $b'_{i} \subseteq \mathrm{acl}^{eq}(b_{i})$ be such that $b'_{i}=\mathrm{Cb}(a_{j}/Cb_{i})$ for $j > i$ and $\{a_{i}, b'_{i}\}_{i < \omega}$ is  a $C$-indiscernible sequence exhibiting instability of $\nind_{C}$ between types of rank at most $n$. Is $\{b'_{i}\}_{i\geq n-1}$ a Morley sequence over $b'_{0}\ldots b'_{n-2}$?

More generally, for $\{a_{i}, b_{i}\}_{i < \omega}$ a $C$-indiscernible sequence exhibiting instability of $\nind_{C}$ between types of rank at most $n$, must there be $a'_{i} \subseteq \mathrm{acl}^{eq}(a_{i})$, $b'_{i} \subseteq \mathrm{acl}^{eq}(b_{i})$ such that $\{a'_{i}, b'_{i}\}_{i < \omega}$ is a $C$-indiscernible sequence exhibiting instability of $\nind_{C}$ between types of rank at most $n$, and $\{b'_{i}\}_{i\geq n-1}$ is a Morley sequence over $b'_{0}\ldots b'_{n-2}$?
    
\end{question}

The proof of Peretz's universal result in rank $3$, Fact \ref{Peretz theorem}, is a relatively ad-hoc, multi-case argument. It is not clear from this argument that his universal results can be recovered in higher ranks by dropping terms, as in the statement of this question.

\textbf{Acknowledgements:} The author would like to thank James Freitag for discussions surrounding geometric stability theory that inspired the use of the implicitly well-known geometric interpretation of stability of the forking relation over a base in finite-rank supersimple theories as a conceptual starting point for the main theorem of this paper. AI was used purely cosmetically in the editing and proofreading of this paper, including in formatting citations, with the exception that the author attempted to use a literature search with Claude Sonnet 5 and Gemini Pro 2.5 to make a decision about the wording and placement of one of the citations in the introduction; the AI systems were unable to find a statement in the literature that the author recalls seeing, but claimed to be able to prove this statement themselves, and the author chose to adopt neutral wording in order to avoid relying on the AI systems' claimed proof.



\bibliographystyle{plain}
\bibliography{refs}

@article{ealy2016thorn,
  author    = {Ealy, Clifton and Onshuus, Alf},
  title     = {\textthorn-Forking and stable forking},
  journal   = {Revista de la Academia Colombiana de Ciencias Exactas, F{\'i}sicas y Naturales},
  volume    = {40},
  number    = {157},
  pages     = {683--689},
  year      = {2016},
  publisher = {Academia Colombiana de Ciencias Exactas, F{\'i}sicas y Naturales}
}

@misc{CastleTrichotomy,
  author = {Castle, Benjamin and Hasson, Assaf and Ye, Jinhe},
  title = {Zilber's Trichotomy in {H}ausdorff Geometric Structures},
  year = {preprint. Available at https://arxiv.org/abs/2405.02209. 2024}
}

@inproceedings{zilber1986structural,
  author    = {Zil'ber, Boris},
  title     = {Structural properties of models of $\aleph_1$-categorical theories},
  booktitle = {Logic, Methodology and Philosophy of Science VII},
  editor    = {Marcus, R. Barcan and et al.},
  series    = {Studies in Logic and the Foundations of Mathematics},
  volume    = {114},
  pages     = {115--128},
  year      = {1986},
  publisher = {North-Holland, Amsterdam}
}

@article{castle2024restricted,
  author    = {Castle, Benjamin},
  title     = {Restricted Trichotomy in Characteristic Zero},
  journal   = {Journal of the American Mathematical Society},
  volume    = {37},
  number    = {4},
  pages     = {1003--1077},
  year      = {2024},
  doi       = {10.1084/S0894-0347-2023-01037-9},
  publisher = {American Mathematical Society}
}

@misc{BrowerHillHyperimaginaries,
  author = {Brower, Donald and Hill, Cameron Donnay},
  title  = {On weak elimination of hyperimaginaries and its consequences},
  year   = {preprint. Available at https://arxiv.org/abs/1210.7883. 2012}
}

@article{casanovaspotier2018,
  author    = {Enrique Casanovas and Joris Potier},
  title     = {Stable Forking and Imaginaries},
  journal   = {Notre Dame Journal of Formal Logic},
  volume    = {59},
  number    = {4},
  pages     = {567--572},
  year      = {2018},
  publisher = {Duke University Press},
  doi       = {10.1215/00294527-2018-0010}
}

@misc{TakahashiConsequences,
  author = {Takahashi, Yuki},
  title = {Consequences of dependent dividing on burden},
  year = {preprint. Available at https://arxiv.org/abs/2511.00282. 2025}
}

@misc{GomezRosyTheories,
  author = {Miguel-G\'{o}mez, Alberto},
  title  = {On stable {K}im-forking and rosy theories},
  year   = {preprint. Available at https://arxiv.org/abs/2503.23936. 2025}
}

@article{HarnikHarrington1984,
  title     = {Fundamentals of forking},
  author    = {Harnik, Victor and Harrington, Leo},
  journal   = {The Journal of Symbolic Logic},
  volume    = {49},
  number    = {4},
  pages     = {1364--1364},
  year      = {1984},
  publisher = {Association for Symbolic Logic}
}

@article{ovchinnikov2025identifiable,
  author  = {Ovchinnikov, Alexey and Pillay, Anand and Pogudin, Gleb and Scanlon, Thomas},
  title   = {Identifiable Specializations for {ODE} Models},
  journal = {Systems \& Control Letters},
  volume  = {204},
  pages   = {106226},
  year    = {2025}
}

@article{meshkat2025algorithm,
  author  = {Meshkat, Nicolette and Ovchinnikov, Alexey and Scanlon, Thomas},
  title   = {Algorithm to Find New Identifiable Reparameterizations of Parametric Rational {ODE} Models},
  journal = {IEEE Transactions on Automatic Control},
  volume  = {70},
  number  = {10},
  pages   = {6688--6703},
  year    = {2025}
}

@article{Li21,
  author  = {Li, Wei and Ovchinnikov, Alexey and Pogudin, Gleb and Scanlon, Thomas},
  title   = {Elimination of Unknowns for Systems of Algebraic Differential-Difference Equations},
  journal = {Transactions of the American Mathematical Society},
  volume  = {374},
  pages   = {303--326},
  year    = {2021},
  doi     = {10.1090/tran/8219}
}

@article{Ovc21,
  author  = {Ovchinnikov, Alexey and Pillay, Anand and Pogudin, Gleb and Scanlon, Thomas},
  title   = {Computing All Identifiable Functions for {ODE} Models},
  journal = {Systems \& Control Letters},
  volume  = {157},
  pages   = {105030},
  year    = {2021},
  doi     = {10.1016/j.sysconle.2021.105030}
}

@article{Li23,
  author  = {Li, Wei and Ovchinnikov, Alexey and Pogudin, Gleb and Scanlon, Thomas},
  title   = {Algorithms Yield Upper Bounds in Differential Algebra},
  journal = {Canadian Journal of Mathematics},
  volume  = {75},
  number  = {1},
  pages   = {29--51},
  year    = {2023},
  doi     = {10.4153/S0008414X21000560}
}

@article{Ovc22,
  author  = {Ovchinnikov, Alexey and Pillay, Anand and Pogudin, Gleb and Scanlon, Thomas},
  title   = {Multi-experiment Parameter Identifiability of {ODEs} and Model Theory},
  journal = {SIAM Journal on Applied Algebra and Geometry},
  volume  = {6},
  number  = {3},
  pages   = {339--367},
  year    = {2022},
  doi     = {10.1137/21M1389845}
}

@article{BuechlerPillayWagner2001,
  author  = {Buechler, Steven and Pillay, Anand and Wagner, Frank O.},
  title   = {Supersimple theories},
  journal = {Journal of the American Mathematical Society},
  year    = {2001},
  volume  = {14},
  pages   = {109--124},
  doi     = {10.1090/S0894-0347-00-00355-1}
}

@article{BenYaacov2004,
  title = {Constructing an Almost Hyperdefinable Group},
  author = {Ben-Yaacov, Itay and Toma{\v{s}}i{\'c}, Ivan and Wagner, Frank O.},
  journal = {Journal of Mathematical Logic},
  volume = {4},
  number = {2},
  pages = {181--212},
  year = {2004},
  doi = {10.1142/S021906130400036X},
  url = {https://www.worldscientific.com/doi/10.1142/S021906130400036X}
}

@article{Tomasic01, author = "Ivan Tomašić", title = "{Geometric Simplicity Theory}", year = "PhD thesis. 2001", url = "https://webspace.maths.qmul.ac.uk/i.tomasic/th.pdf" }

@unpublished{Casanovas2006,
  author = {Casanovas, Enrique},
  title  = {Pregeometries and minimal types},
  note   = {Lecture notes, Universitat de Barcelona},
  year   = {2006. Available at https://www.ub.edu/modeltheory/documentos/pregeometries.pdf}
}

@article{baldwin1971strongly,
  author    = {Baldwin, John T. and Lachlan, Alistair H.},
  title     = {On strongly minimal sets},
  journal   = {The Journal of Symbolic Logic},
  volume    = {36},
  number    = {1},
  pages     = {79--96},
  year      = {1971},
  publisher = {Cambridge University Press}
}

@misc{HansonWiki,
  author = {Hanson (ed.), James},
  title = {Lascar rank},
  year = {mirror of Model Theory Fandom, online encyclopedia page. Available at https://james-hanson.github.io/wiki/Lascar{\textunderscore}rank.html.},
}

@article{Palacin2013Ample,
  title = {Ample thoughts},
  author = {Palacín, D. and Wagner, F. O.},
  journal = {Journal of Symbolic Logic},
  volume = {78},
  number = {2},
  pages = {489--510},
  year = {2013},
  doi = {10.2178/jsl.7802080},
  url = {https://doi.org}
}

@misc{Bos25,
  author = {Bossut, Yvon},
  title = {A Note on Stable {K}im-forking},
  year = {Preprint. Available at https://arxiv.org. 2025},
}

@misc{LongKoponenConjecturePreprint,
author = {John Baldwin and James Freitag  and Scott Mutchnik},
title = {Simple Homogeneous Structures and Indiscernible Sequence Invariants},
year = {preprint. Available at https://arxiv.org/abs/2405.08211. 2024}
}

@article{HHL00,
 ISSN = {0003486X, 19398980},
 URL = {http://www.jstor.org/stable/2661382},
 author = {Bradd Hart and Ehud Hrushovski and Michael C. Laskowski},
 journal = {Annals of Mathematics},
 number = {1},
 pages = {207--257},
 publisher = {[Annals of Mathematics, Trustees of Princeton University on Behalf of the Annals of Mathematics, Mathematics Department, Princeton University]},
 title = {The Uncountable Spectra of Countable Theories},
 urldate = {2025-08-09},
 volume = {152},
 year = {2000}
}

@phdthesis{brower2012,
  title     = {Aspects of stability in simple theories},
  author    = {Brower, Donald A.},
  school    = {University of Notre Dame},
  year      = {2012},
  address   = {Notre Dame, IN},
  type      = {Doctoral dissertation},
  url       = {https://curate.nd.edu/articles/thesis/Aspects_of_stability_in_simple_theories/24737640}
}

@article{Per06,
	author = {Assaf Peretz},
	doi = {10.2178/jsl/1140641179},
	journal = {Journal of Symbolic Logic},
	number = {1},
	pages = {347--359},
	publisher = {Association for Symbolic Logic},
	title = {Geometry of Forking in Simple Theories},
	volume = {71},
	year = {2006}
}

@article{Kim01,
	author = {Byunghan Kim},
	doi = {10.2307/2695047},
	journal = {Journal of Symbolic Logic},
	number = {2},
	pages = {822--836},
	publisher = {Association for Symbolic Logic},
	title = {Simplicity, and Stability in There},
	volume = {66},
	year = {2001}
}

@article{palacin2012omega,
  title={On $\omega$-categorical simple theories},
  author={Palac{\'\i}n, Daniel},
  journal={Archive for Mathematical Logic},
  volume={51},
  pages={709--717},
  year={2012},
  publisher={Springer}
}

@inproceedings{Kim10,
  title={THE GROUP CONFIGURATION THEOREM AND ITS APPLICATIONS},
  author={Byunghan Kim},
  year={2010},
  url={https://api.semanticscholar.org/CorpusID:213176310}
}

@article{BTW02,
author = {Itay Ben-Yaacov and Ivan Tomasic and Frank O. Wagner},
title = {{The group configuration in simple theories and its applications}},
volume = {8},
journal = {Bulletin of Symbolic Logic},
number = {2},
publisher = {Association for Symbolic Logic},
pages = {283 -- 298},
year = {2002},
doi = {10.2178/bsl/1182353874},
URL = {https://doi.org/10.2178/bsl/1182353874}
}

@article{TW03,
  title={Applications of the Group Configuration Theorem in Simple Theories},
  author={Ivan Tomašić and Frank O. Wagner},
  journal={J. Math. Log.},
  year={2003},
  volume={3},
  url={https://api.semanticscholar.org/CorpusID:59528831}
}

@article{Hrustrongmin, AUTHOR = " Hrushovski,E.", TITLE = " A new strongly
minimal set", YEAR = "1993", VOLUME = "62", JOURNAL = " Annals of Pure and
Applied Logic", pages = "147-166" }

@article{PW13,
	doi = {10.1215/00294527-2143970},
  
	url = {https://doi.org/10.1215%2F00294527-2143970},
  
	year = 2013,
  
	publisher = {Duke University Press},
  
	volume = {54},
  
	number = {3-4},
  
	author = {D. Palac\'in and F. O. Wagner},
  
	title = {Elimination of Hyperimaginaries and Stable Independence in Simple {CM}-Trivial Theories},
  
	journal = {Notre Dame Journal of Formal Logic}
}

@article{K21, title={Weak canonical bases in $\mathrm{NSOP}_1$ theories}, volume={86}, DOI={10.1017/jsl.2021.45}, number={3}, journal={The Journal of Symbolic Logic}, publisher={Cambridge University Press}, author={Byunghan Kim}, year={2021}, pages={1259–1281}}

@book{Wag00,
  title={Simple Theories},
  author={Wagner, Frank},
  year={2000},
  publisher={Springer Dordrecht}
}

@misc{Sh99,
  doi = {10.48550/ARXIV.MATH/9910158},
  
  url = {https://arxiv.org/abs/math/9910158},
  
  author = {Shelah, Saharon},
  
  title = "{On what I do not understand (and have something to say), model theory}",
  
  publisher = {arXiv},
  
  year = {Preprint. Available at https://arxiv.org/abs/math/9910158. 1999},
  
  copyright = {Assumed arXiv.org perpetual, non-exclusive license to distribute this article for submissions made before January 2004}
}

@article{HKP00,
 ISSN = {00224812},
 URL = {http://www.jstor.org/stable/2586538},
 author = {Bradd Hart and Byunghan Kim and Anand Pillay},
 journal = {The Journal of Symbolic Logic},
 number = {1},
 pages = {293--309},
 publisher = {Association for Symbolic Logic},
 title = {Coordinatisation and Canonical Bases in Simple Theories},
 urldate = {2023-03-12},
 volume = {65},
 year = {2000}
}

@article{Che14,
  title={Theories without the tree property of the second kind},
  author={Artem Chernikov},
  journal={Ann. Pure Appl. Log.},
  year={2014},
  volume={165},
  pages={695-723}
}

@book{Kim14,
  title={Simplicity theory},
  author={Kim, Byunghan},
  year={2014},
  publisher={Oxford University Press}
}

@article{Bu91,
 ISSN = {00224812},
 URL = {http://www.jstor.org/stable/2275467},
 author = {Steven Buechler},
 journal = {The Journal of Symbolic Logic},
 number = {4},
 pages = {1184--1194},
 publisher = {[Association for Symbolic Logic, Cambridge University Press]},
 title = {Pseudoprojective Strongly Minimal Sets are Locally Projective},
 urldate = {2022-09-07},
 volume = {56},
 year = {1991}
}

@article{KP99,
title = {Simple theories},
journal = {Annals of Pure and Applied Logic},
volume = {88},
number = {2},
pages = {149-164},
year = {1997},
note = {Joint AILA-KGS Model Theory Meeting},
issn = {0168-0072},
doi = {https://doi.org/10.1016/S0168-0072(97)00019-5},
url = {https://www.sciencedirect.com/science/article/pii/S0168007297000195},
author = {Byunghan Kim and Anand Pillay}
}

@article{Ons06,
 ISSN = {00224812},
 URL = {http://www.jstor.org/stable/27588433},
 author = {Alf Onshuus},
 journal = {The Journal of Symbolic Logic},
 number = {1},
 pages = {1--21},
 publisher = {[Association for Symbolic Logic, Cambridge University Press]},
 title = {Properties and Consequences of Thorn-Independence},
 urldate = {2022-08-18},
 volume = {71},
 year = {2006}
}

@book{P96,
  title={Geometric Stability Theory},
  author={Pillay, A.},
  isbn={9780198534372},
  lccn={lc96017222},
  series={Oxford logic guides},
  url={https://books.google.com/books?id=aRb6wAEACAAJ},
  year={1996},
  publisher={Clarendon Press}
}

@article{KP98,
    author = {Kim, Byunghan},
    title = "{Forking in simple unstable theories}",
    journal = {Journal of the London Mathematical Society},
    volume = {57},
    number = {2},
    pages = {257-267},
    year = {1998},
    month = {04},
    issn = {0024-6107},
    doi = {10.1112/S0024610798005985},
    url = {https://doi.org/10.1112/S0024610798005985},
    eprint = {https://academic.oup.com/jlms/article-pdf/57/2/257/2534631/57-2-257.pdf},
}

@misc{FD, author={Conant, Gabriel}, year={forkinganddividing.com, web media.}}

@book{Sh90,
  title={Classification theory: and the number of non-isomorphic models},
  author={Shelah, Saharon},
  year={1990},
  publisher={Elsevier}
}

@article{M65,
  title={Categoricity in power},
  author={Morley, Michael},
  journal={Transactions of the American Mathematical Society},
  volume={114},
  number={2},
  pages={514--538},
  year={1965},
  publisher={JSTOR}
}

@article{kim2001around,
  title={Around stable forking},
  author={Kim, Byunghan and Pillay, Anand},
  journal={Fundamenta Mathematicae},
  volume={170},
  number={1-2},
  pages={107--118},
  year={2001},
  publisher={Citeseer}
}

@article{BYC14, title={An independence theorem for $\mathrm{NTP}_{2}$ theories}, volume={79}, DOI={10.1017/jsl.2013.22}, number={1}, journal={The Journal of Symbolic Logic}, publisher={Cambridge University Press}, author={Yaacov, Itaï Ben and Chernikov, Artem}, year={2014}, pages={135–153}}

@article{D19, title={Forking, imaginaries and other features of $\text {ACFG}$}, volume={86}, DOI={10.1017/jsl.2021.34}, number={2}, journal={The Journal of Symbolic Logic}, publisher={Cambridge University Press}, author={D'Elbée, Christian}, year={2021}, pages={669–700}}

@article{She95,
title = {Toward classifying unstable theories},
journal = {Annals of Pure and Applied Logic},
volume = {80},
number = {3},
pages = {229-255},
year = {1996},
issn = {0168-0072},
doi = {https://doi.org/10.1016/0168-0072(95)00066-6},
url = {https://www.sciencedirect.com/science/article/pii/0168007295000666},
author = {Saharon Shelah}
}

@article{SU08,
title = {More on $\mathrm{SOP}_{1}$ and $\mathrm{SOP}_{2}$},
journal = {Annals of Pure and Applied Logic},
volume = {155},
number = {1},
pages = {16-31},
year = {2008},
issn = {0168-0072},
doi = {https://doi.org/10.1016/j.apal.2008.02.003},
url = {https://www.sciencedirect.com/science/article/pii/S0168007208000250},
author = {Saharon Shelah and Alexander Usvyatsov}
}

@article{DS04,
	year = {2004},
	journal = {Annals of Pure and Applied Logic},
	title = {On $\lhd^{*}$-Maximality},
	doi = {10.1016/j.apal.2003.11.001},
	number = {1-3},
	volume = {125},
	publisher = {Elsevier},
	author = {Mirna D\v{z}amonja and Saharon Shelah},
	pages = {119--158}
}

\end{document}